\documentclass[12pt]{article}

\usepackage[margin=1in]{geometry}
\RequirePackage{amsthm,amsmath,amsfonts,amssymb}
\RequirePackage[authoryear]{natbib}
\usepackage{mathrsfs}
\usepackage{amsbsy,graphicx,subcaption}
\usepackage{verbatim}
\usepackage[ruled,vlined]{algorithm2e}
\usepackage[colorlinks,citecolor=blue,urlcolor=blue]{hyperref}

\newcommand{\df}{\mathrm{d}}
\newcommand{\PR}{\mathcal{R}}
\newcommand{\PD}{\mathcal{D}}
\newcommand{\PS}{\mathcal{S}}
\newcommand{\PM}{\mathcal{M}}

\newcommand{\VD}{V_{\scriptsize\mbox{D}}}
\newcommand{\VR}{V_{\scriptsize\mbox{R}}}
\newcommand{\VM}{V_{\scriptsize\mbox{M}}}
\newcommand{\VS}{V_{\scriptsize\mbox{S}}}
\newcommand{\rhoD}{\rho_{\scriptsize\mbox{D}}}
\newcommand{\rhoR}{\rho_{\scriptsize\mbox{R}}}
\newcommand{\rhoM}{\rho_{\scriptsize\mbox{M}}}
\newcommand{\rhoS}{\rho_{\scriptsize\mbox{S}}}
\newcommand{\SD}{\Sigma_{\scriptsize\mbox{D}}}
\newcommand{\SR}{\Sigma_{\scriptsize\mbox{R}}}
\newcommand{\SM}{\Sigma_{\scriptsize\mbox{M}}}
\newcommand{\SigS}{\Sigma_{\scriptsize\mbox{S}}}
\newcommand{\MRspace}{M_{\scriptsize\mbox{R}}}

\newcommand{\X}{\mathsf{X}}

\newcommand{\B}{\mathcal{B}}

\newcommand{\pcite}[1]{\citeauthor{#1}'s \citeyearpar{#1}}

\theoremstyle{plain}
\newtheorem{theorem}{Theorem}[section]
\newtheorem{lemma}[theorem]{Lemma}
\newtheorem{corollary}[theorem]{Corollary}
\newtheorem{proposition}[theorem]{Proposition}

\newtheorem{example}[theorem]{Example}
\newtheorem{remark}[theorem]{Remark}

\title{\bf Analysis of two-component Gibbs samplers using the theory of two projections}
\author{Qian Qin \\ School of Statistics \\ University of Minnesota}
\date{}

\begin{document}

	\maketitle
		
		\begin{abstract}
			The theory of two projections is utilized to study two-component Gibbs samplers.
			Through this theory, previously intractable problems regarding the asymptotic variances of two-component Gibbs samplers are reduced to elementary matrix algebra exercises.
			It is found that in terms of asymptotic variance, the two-component random-scan Gibbs sampler is never much worse, and could be considerably better than its deterministic-scan counterpart, provided that the selection probability is appropriately chosen.
			This is especially the case when there is a large discrepancy in computation cost between the two components. 
			The result contrasts with the known fact that the deterministic-scan version has a faster convergence rate, which can also be derived from the method herein. 
			On the other hand, a modified version of the deterministic-scan sampler that accounts for computation cost can outperform the random-scan version.
		\end{abstract}
		
%		\begin{keyword}[class=MSC]
%			\kwd[Primary ]{60J05}
%		\end{keyword}
%		
%		\begin{keyword}
%			\kwd{Asymptotic variance}
%			\kwd{Convergence rate}
%			\kwd{Markov operator}
%			\kwd{Matrix representation}
%			\kwd{MCMC}
%		\end{keyword}

	\section{Introduction}
	
	Gibbs samplers are a class of Markov chain Monte Carlo (MCMC) algorithms commonly used in statistics for sampling from intractable distributions \citep{gelfand1990sampling,casella1992explaining}.
	In this work, I will introduce a method for analyzing different variants of two-component Gibbs samplers via the theory of two projections developed by \cite{halmos1969two}.
	As a first application of this new method, I will conduct a detailed comparison between deterministic-scan and random-scan samplers, which I now define.

	Let $(\X_1,\B_1)$ and $(\X_2,\B_2)$ be measurable spaces and let $(\X,\B) = (\X_1 \times \X_2, \B_1 \times \B_2)$.
	Let~$\pi$ be a probability measure on $(\X,\B)$ that is the joint distribution of the random element $(X_1,X_2)$, where $X_i$ is $\X_i$-valued for $i=1,2$.
%	Assume that regular conditional probabilities exist for $X_1$ and $X_2$.
	For $x_1 \in \X_1$ and $x_2 \in \X_2$, let $\pi_2(\cdot\mid x_2)$ be the conditional distribution of $X_1\mid X_2=x_2$, and let $\pi_1(\cdot\mid x_1)$ be that of $X_2\mid X_1=x_1$.
%	In other words, $\pi(\df x_1, \df x_2) = \pi_1(\df x_2 \mid x_1) \mu_1(\df x_1) = \pi_2(\df x_1 \mid x_2) \mu_2(\df x_2)$, where $\mu_1$ and $\mu_2$ are the marginal distributions of $X_1$ and $X_2$, respectively.
	When one cannot sample from~$\pi$ directly, but can sample from $\pi_1(\cdot|x_1)$ and $\pi_2(\cdot|x_2)$, a two-component Gibbs sampler may be used to produce an approximate sample from~$\pi$.
	Although simple, Gibbs algorithms with two components are surprisingly useful in practice \citep{tanner1987calculation,albert1993bayesian,polson2013bayesian}.
	%	In particular, data augmentation algorithms can be viewed as a variant of the two-component Gibbs sampler.
	There are two basic forms of two-component Gibbs samplers.
	Each simulates a Markov chain that has~$\pi$ as a stationary distribution.
	Moreover, under mild conditions, the marginal distribution of the $t$th element of each Markov chain converges to~$\pi$ in some sense as $t \to \infty$ \citep{tierney1994markov,roberts2006harris}.

	The first type, called deterministic-scan Gibbs (DG) sampler, simulates a time-inhomogeneous Markov chain $(\tilde{X}_t)$ in the following fashion.
	($T$ here is the length of the simulation.)
	
	\bigskip
	
	\begin{algorithm}[H]
		\caption{DG sampler} \label{alg:dg}
		Draw $\tilde{X}_0 = (X_{1,0},X_{2,0})$ from some initial distribution on $(\X, \B)$, and
		set $t=0$\;
		
		\While{$t < T$}{
			\If{$t=2s$ for some non-negative integer~$s$}{
				draw $X_{2,t+1}$ from $\pi_1(\cdot\mid X_{1,t})$, set $X_{1,t+1} = X_{1,t}$, and let $\tilde{X}_{t+1} = (X_{1,t+1},X_{2,t+1})$\;
			}
			\If{$t=2s+1$ for some non-negative integer~$s$}{
				draw $X_{1,t+1}$ from $\pi_2(\cdot\mid X_{2,t})$, set $X_{2,t+1} = X_{2,t}$, and let $\tilde{X}_{t+1} = (X_{1,t+1},X_{2,t+1})$\;
			}
			set $t=t+1$ \;
		}
	\end{algorithm}
	
	\bigskip
	
	\begin{remark}
		A time-homogeneous version of the DG chain can be obtained through thinning.
		Indeed, it is common to discard $\tilde{X}_t$ when~$t$ is odd, and only use $(\tilde{X}_{2s})_{s=0}^{\lfloor T/2 \rfloor}$ as a Monte Carlo sample.
		However, thinning reduces the efficiency of the sampler \citep{maceachern1994subsampling}, and is often discouraged \citep{link2012thinning}.
	\end{remark}

	The second type is called the random-scan Gibbs (RG) sampler.
	To run the algorithm, one needs to specify a selection probability $r \in (0,1)$.
	The algorithm then simulates a time-homogeneous Markov chain $(\tilde{X}_t)$ in the following fashion.
	
	\bigskip
	
	\begin{algorithm}[H]
		\caption{RG sampler}
		Draw $\tilde{X}_0 = (X_{1,0},X_{2,0})$ from some initial distribution on $(\X, \B)$, and
		set $t=0$ \;
		
		\While{$t < T$}{
			draw $W$ from a $\mbox{Bernoulli}(r)$ distribution\;
			\If{$W=0$}{
				draw $X_{2,t+1}$ from $\pi_1(\cdot\mid X_{1,t})$, and set $X_{1,t+1}=X_{1,t}$\;
			}
			\If{$W=1$}{
				draw $X_{1,t+1}$ from $\pi_2(\cdot\mid X_{2,t})$, and set $X_{2,t+1}=X_{2,t}$\;
			}
			set $\tilde{X}_{t+1} = (X_{1,t+1},X_{2,t+1})$\;
			set $t=t+1$\;
		}
	\end{algorithm}
	
	\bigskip
	
	Both samplers generate new elements by updating $X_1$ using the conditional distribution of $X_1 \mid X_2$, and updating $X_2$ using the conditional distribution of $X_2 \mid X_1$.
	They only differ in terms of which component is updated in each iteration.
	%	Both chains have~$\pi$ as a stationary distribution.
	
	One important question regarding the DG and RG samplers is which of them performs better.
	This problem is not unique to the two-component case, but it is in this case where substantial progress has been made, as I now describe.

	When comparing MCMC algorithms, there are two main aspects to consider: convergence speed and asymptotic variance \citep{jones2001honest}.
	Moreover, one must also account for computation cost, i.e., the time it takes to run one iteration of each algorithm.
	\cite{qin2021convergence} showed that the DG algorithm is better than the RG algorithm in terms of $L^2$ convergence rate.
	Their result takes computation time into account, and holds for every selection probability $r \in (0,1)$.
	As to asymptotic variance, existing comparisons are less conclusive.
	
	Consider a generic MCMC algorithm that simulates a Markov chain $(\tilde{X}_t)_{t=0}^{\infty}$ which converges to its stationary distribution~$\pi$.
	Then $(\tilde{X}_0,\dots,\tilde{X}_{T-1})$ forms an approximate Monte Carlo sample from~$\pi$.
	The Monte Carlo sample is usually used to estimate the mean of a function $f: \X \to \mathbb{R}$, i.e., $\pi f := \int_{\X} f(x) \pi(\df x)$.
	The usual estimator is the sample mean
	\[
	S_T(f) := \frac{1}{T} \sum_{t=0}^{T-1} f(\tilde{X}_t).
	\]
	Under regularity conditions, $S_T(f)$ is subject to the central limit theorem (CLT):
	\[
	\sqrt{T}[S_T(f) - \pi f] \xrightarrow{d} \mbox{N}[0, V(f)] \quad \mbox{as } T \to \infty.
	\]
	See \cite{dobrushin1956central}, \cite{greenwood1998information}, and \cite{jones2004markov} for Markov chain CLTs under various settings.
	$V(f)$ is called the asymptotic variance, and can usually be obtained from the formula 
	\[
	V(f) = \lim_{T \to \infty} T \, \mbox{var} [S_T(f)] = \lim_{T \to \infty} T^{-1} \sum_{t=0}^{T-1} \sum_{t'=0}^{T-1} \mbox{cov}[ f(\tilde{X}_t), f(\tilde{X}_{t'})] ,
	\]
	where $\tilde{X}_0$ can be assumed to follow~$\pi$.
	Disregarding computation cost, one may say the smaller $V(f)$ the better.
	Define $\VD(f)$ and $\VR(f,r)$ to be the asymptotic variances associated with~$f$ for the DG and RG sampler, respectively, where~$r$ is the selection probability.
	Exact formulas for these quantities will be given in Section~\ref{ssec:basic}.
	\cite{greenwood1998information} showed that $\VD(f) \leq \VR(f,1/2)$.
	See also \cite{andrieu2016random} where the result is extended to beyond Gibbs algorithms.
	As far as I am aware, no similar result existed for $r \neq 1/2$.
	Moreover, the arguments in \cite{greenwood1998information} and \cite{andrieu2016random} heavily rely on the symmetry of the RG sampler that arises only when $r = 1/2$, and it seems unlikely that they can be extended to the $r \neq 0.5$ case.
	This motivates the current work.

	To appreciate the potential benefits of using a selection probability~$r$ that is not $1/2$, one needs to think about computation cost.
	Consider this:
	If the time it takes to draw from $\pi_2(\cdot|x_2)$ is much longer than $\pi_1(\cdot|x_1)$ for $x_1 \in \X_1$ and $x_2 \in \X_2$, then one has an incentive to update $X_2$ using $\pi_1(\cdot|x_1)$ more frequently, and may want to use an RG sampler with a small~$r$ rather than the DG sampler.
	To conduct a concrete analysis, I will study the adjusted asymptotic variance, defined below.
	
	Suppose that on average, the aforementioned generic MCMC algorithm takes $\tau_0$ units of time to produce a sample point, so that after running the algorithm for $\tilde{\tau}$ units of time where $\tilde{\tau}$ is large, around $ \tilde{\tau}/\tau_0$ sample points are generated. 
	Then the sample mean has variance roughly equal to $\tau_0 V(f)/\tilde{\tau}$.
	Define the computation time adjusted asymptotic variance to be $V^{\dagger}(f) := \tau_0 V(f)$.
	Given a fixed amount of computation effort, $V^{\dagger}(f)$, instead of $V(f)$, is what we should focus on.
	Define $\VD^{\dagger}(f)$ and $\VR^{\dagger}(f,r)$ to be the adjusted asymptotic variance for the DG and RG samplers, respectively.
	It will be shown that when~$r$ is well chosen, $\VR^{\dagger}(f,r)$ can be much smaller than $\VD^{\dagger}(f)$ for certain functions~$f$, especially when the costs of drawing from $\pi_1$ and $\pi_2$ differ significantly.
	Moreover, for any function~$f$ such that $f(X)$ has finite second moment for $X \sim \pi$, under mild conditions, $\VR^{\dagger}(f,r) \leq 2 \VD^{\dagger}(f)$, given that~$r$ is well chosen.
	The appropriate value of~$r$ depends on the time it takes to draw from the two conditional distributions, and an explicit formula will be provided.
	These results mean that the RG sampler can outperform the DG algorithm by a large margin in some scenarios, while never being too much worse.
	This is in contrast with \pcite{qin2021convergence} result on convergence rate.
	However, as discussed in Section~\ref{ssec:da}, the DG sampler could still compete with RG in terms of adjusted asymptotic variance for certain problems.
	
	The DG sampler can be made more robust through a simple modification.
	Suppose that it is much less costly draw from $\pi_1$ compared to $\pi_2$.
	Then one may consecutively draw from $\pi_1$ many times before drawing once from $\pi_2$.
	This modified sampler, defined in Section~\ref{ssec:modified}, behaves similarly to an RG sampler with some well-chosen~$r$ in terms of adjusted asymptotic variance.
	Moreover, it is possible to parallelize parts of the modified DG algorithm to make it even more efficient.

	Aside from acquiring the more specific results described above, a central goal of this paper is to demonstrate how some tools in classical linear algebra can trivialize difficult problems concerning two-component Gibbs.
	The key is \pcite{halmos1969two} theory of two projection operators.
	It is known that the two conditional distributions in a two-component Gibbs sampler correspond to two orthogonal projections on some function space \citep{greenwood1998information,diaconis2010stochastic}.
	\cite{halmos1969two} gave block matrix representations for any given pair of orthogonal projections.
	Using these representations, one can obtain explicit formulas for the asymptotic variance of a given Gibbs sampler.
	The framework is powerful because it reduces the problem at hand to simple matrix calculations.
	I am unaware of previous works that analyze Gibbs samplers using \pcite{halmos1969two} theory.
	The tools developed here can also be utilized to study the convergence rates of Gibbs chains, and reproduce \pcite{qin2021convergence} result.
	Moreover, they can be used to study other variants of two-component Gibbs samplers, as demonstrated in the Supplement.
	Indeed, there are many potentially interesting variants of two-component Gibbs sampler besides the ones discussed here, and \pcite{halmos1969two} theory opens an avenue for studying them in a manner that was not possible before.
	
%	To facilitate convergence analysis within the above framework, I will derive some general formulas that characterize the convergence rate of a Markov chain in terms of its operator, defined through its transition kernel.
%	Compared to existing results of similar flavor \citep[e.g.,][]{roberts1997geometric}, results herein cover chains that are non-reversible and/or time-inhomogeneous, which have received limited attention in the MCMC literature.
	
	The rest of the article is organized as follows.
	Section~\ref{sec:pre} provides the necessary preliminaries, including a short introduction to \pcite{halmos1969two} theory.
	Section~\ref{sec:main} illustrates the usefulness of \pcite{halmos1969two} theory via a comparison between the (standard and modified) DG and RG samplers in terms of adjusted asymptotic variance.
	A comparison concerning convergence rate is conducted in Section~\ref{sec:rate}.
	Also included in this section are some general formulas involving Markov chain convergence rate as described in the previous paragraph.
	Section~\ref{sec:discuss} contains some discussion.
	Some technical details are relegated to the Appendix and Supplement.
	The Supplement also contains an analysis of a fourth type of two-component Gibbs sampler using the theory of two projections.

	\section{Preliminaries} \label{sec:pre}
	
	\subsection{Basic properties of two-component Gibbs samplers} \label{ssec:basic}

	The transition laws of the DG and RG Markov chains are described by their Markov transition kernels, or Mtks.
	In general, an Mtk on $(\X,\B)$ is a function $K:\X \times \B \to [0,1]$ such that $K(x,\cdot)$ is a probability measure for $x \in \X$, and $K(\cdot,A)$ is a measurable function for $A \in \B$.
	%	For a time-homogeneous Markov chain, say $(\tilde{X}_t)$, the associated (single-step) Mtk $K(x,\cdot)$ gives the conditional distribution of $\tilde{X}_{t+1}$ given $\tilde{X}_t$ evaluated at $\tilde{X}_t = x$.
	%	For a time-inhomogeneous chain, its transition law can be defined by different Mtks at different time points.
	If~$K$ and~$G$ are Mtks on $(\X,\B)$, then their mixture $a K + (1-a) G$ with $a \in [0,1]$ defines an Mtk such that
	\[
	[aK+(1-a)G](x,A) = a K(x,A) + (1-a) G(x,A), \quad \forall x \in \X, A \in \B,
	\]
	and their product $KG$ is defined to be an Mtk such that
	\[
	(KG)(x, A) = \int_{\X} K(x,\df x') G(x',A), \quad \forall x \in \X, A \in \B.
	\]
	For a non-negative integer~$t$, the $t$th power of an Mtk~$K$, $K^t$, is defined to be $\prod_{s=1}^t K$, so that $K^1 = K$, and $K^0(x,\cdot)$ is the point mass at~$x$.
	For a Markov chain $(\tilde{X}_t)$ that is possibly time-inhomogeneous (e.g., the DG chain), its $(t',t'+t)$-Mtk, where~$t$ and~$t'$ are non-negative integers, is an Mtk $K_{t',t'+t}$ such that $K_{t',t'+t}(x,\cdot)$ gives the conditional distribution of $\tilde{X}_{t'+t}$ given $\tilde{X}_{t'}=x$.
	In particular, call $K_t := K_{0,t}$ the $t$-step Mtk.
	If the chain is time-homogeneous with a single-step Mtk~$K$, then $K_{t',t'+t} = K_t = K^t$.
	If the distribution of $\tilde{X}_0$ is given by some probability measure~$\mu$, then
	\[
	(\mu K_t)(\cdot) := \int_{\X} \mu(\df x) K_t(x,\cdot)
	\]
	gives the marginal distribution of $\tilde{X}_t$.
%	If the marginal distribution of $\tilde{X}_{t'}$ is given by some probability measure~$\mu$, then
%	\[
%	(\mu K_{t',t'+t})(\cdot) := \int_{\X} \mu(\df x) K_{t',t'+t}(x,\cdot)
%	\]
%	gives the marginal distribution of $\tilde{X}_{t'+t}$.
%	In particular, $\mu K_t$ gives the marginal distribution of $\tilde{X}_t$ when the initial distribution is~$\mu$.

	The conditional distribution $\pi_1(\cdot \mid \cdot)$ is assumed to have the following standard properties:
	For $x_1 \in \X_1$, $\pi_1(\cdot \mid x_1)$ is a probability measure on $\B_2$, and for $B \in \B_2$, $\pi_1(B \mid \cdot)$ is a measurable function on $\X_1$.
	Moreover, if $\mu_1$ is the marginal distribution of $X_1$, then, for any non-negative measurable $f: \X \to [0,\infty)$,
	\[
	\int_{\X} f(x_1,x_2) \, \pi(\df (x_1,x_2)) = \int_{\X_1} \left[ \int_{\X_2} f(x_1, x_2) \, \pi_1(\df x_2 \mid x_1) \right] \mu_1(\df x_1). 
	\]
	The conditional distribution $\pi_2$ is assumed to have analogous properties.
	
	Let us now write down the Mtks of two-component Gibbs samplers.
	For a point~$x$ in a generic measurable space, use $\delta_x(\cdot)$ to denote the point mass concentrated at~$x$.
	Define, for $(x_1,x_2) \in \X$ and $A \in \B$,
	\begin{equation} \label{eq:P1P2}
		\begin{aligned}
			P_1((x_1,x_2),A) &= \int_A \pi_1(\df x_2'\mid x_1) \delta_{x_1}(\df x_1'), \\
			P_2((x_1,x_2),A) &= \int_A \pi_2(\df x_1'\mid x_2) \delta_{x_2}(\df x_2').
		\end{aligned}
	\end{equation}
	Then $P_1$ characterizes the transition rule for updating $X_2$ through the conditional distribution of $X_2\mid X_1$, and $P_2$ does the same for updating $X_1$ through $X_1\mid X_2$.
	In other words, for $i \in \{1,2\}$, $P_i$ leaves $X_i$ the same and updates the other component.
	
	Consider first the Markov chain associated with the DG algorithm.
	Let $\PD_{t',t'+t}$ be its $(t',t'+t)$-Mtk, and let $\PD_t = \PD_{0,t}$.
	Then, for a non-negative integer~$s'$ and positive integer~$s$, the following hold:
	\begin{equation} \label{eq:PD}
		\begin{aligned}
			\PD_{2s',2(s'+s)-1} = \PD_{2s-1} = (P_1P_2)^{s-1} P_1, &\quad
			\PD_{2s',2(s'+s)} = \PD_{2s} = (P_1P_2)^s, \\
			\PD_{2s'+1,2(s'+s)} = \PD_{1,2s} = (P_2P_1)^{s-1} P_2, &\quad
			\PD_{2s'+1,2(s'+s)+1} = \PD_{1,2s+1} = (P_2P_1)^s.
		\end{aligned}
	\end{equation}
	On the other hand, the RG sampler simulates a time-homogeneous Markov chain, whose $t$-step Mtk is
	\begin{equation} \label{eq:PR}
		\PR_t := [(1-r)P_1 + rP_2]^t.
	\end{equation}

	One of the main goals of the current work is to conduct a comparison between the two algorithms in terms of asymptotic variance.
	As a shorthand notation, for any signed measure~$\mu$ and measurable function~$f$ on $(\X,\B)$, let $\mu f = \int_{\X} f(x) \mu(\df x)$ whenever the integral is well-defined.
	Fix a function $f: \X \to \mathbb{R}$ such that $\pi f^2 < \infty$, where $f^2(x) := f(x)^2$ for $x \in \X$.
	Consider the asymptotic variance of $S_T(f)$, where $S_T(f)$ is the Monte Carlo sample mean based on either the DG or RG sampler; see the Introduction.
	Without loss of generality, assume that $\pi f = 0$.
	Consider first the DG sampler.
	Denote by $(\tilde{X}_t)$ a DG Markov chain with $\tilde{X}_0 \sim \pi$.
	By Lemma 24 in \cite{maire2014comparison}, the asymptotic variance can be calculated as follows:
	\begin{equation} \label{eq:VD-1}
		\begin{aligned}
			\VD(f) 
%			=& \mbox{var}[f(\tilde{X}_0)] + \sum_{t=1}^{\infty} \mbox{cov}[f(\tilde{X}_0), f(\tilde{X}_t)] + \sum_{t=1}^{\infty} \mbox{cov}[f(\tilde{X}_1), f(\tilde{X}_{t+1})] \\
%			=& E[f(\tilde{X}_0)^2] + \sum_{t=1}^{\infty} E[f(\tilde{X}_0) f(\tilde{X}_t)] + \sum_{t=1}^{\infty} E[f(\tilde{X}_1) f(\tilde{X}_{t+1})] \\
			=& E[f(\tilde{X}_0)^2] + \sum_{s=1}^{\infty} E[f(\tilde{X}_0) f(\tilde{X}_{2s-1})] + \sum_{s=1}^{\infty} E[f(\tilde{X}_0) f(\tilde{X}_{2s})] \\
			& \qquad \qquad \,\, + \sum_{s=1}^{\infty} E[f(\tilde{X}_1) f(\tilde{X}_{2s})] + \sum_{s=1}^{\infty} E[f(\tilde{X}_1) f(\tilde{X}_{2s+1})].
		\end{aligned}
	\end{equation}
	See also \cite{greenwood1998information}.
	For the RG algorithm with selection probability $r \in (0,1)$, which simulates a time-homogeneous Markov chain, the asymptotic variance is a lot simpler.
	See, e.g., \cite{jones2004markov}, Corollary~1.
	Let $(\tilde{X}_t)$ instead be an RG chain with $\tilde{X}_0 \sim \pi$.
	Then the asymptotic variance is
	\begin{equation} \label{eq:VR-1}
		\begin{aligned}
			\VR(f,r) 
%			=& \mbox{var}[f(\tilde{X}_0)] + 2\sum_{t=1}^{\infty} \mbox{cov}[f(\tilde{X}_0), f(\tilde{X}_t)] \\
			=& E[f(\tilde{X}_0)^2] + 2 \sum_{t=1}^{\infty} E[f(\tilde{X}_0)  f(\tilde{X}_t)].
		\end{aligned}
	\end{equation}
	
	These two quantities will be analyzed in detail in the next subsection.
	As mentioned in the Introduction, one needs to take into account the computation costs of the two algorithms.
	Suppose that the time it takes to sample from $\pi_1(\cdot\mid x_1)$ is a constant (i.e., independent of $x_1$), and set this to be unit time.
	Suppose that the time it takes to sample from $\pi_2(\cdot\mid x_2)$ is also a constant, denoted by~$\tau$.
	Then, on average, the DG algorithm takes $(1+\tau)/2$ units of time to produce a sample point, whereas the RG algorithm takes $r \tau + 1-r$ to do so.
	Although the above assumption may not always hold in practice, the simplification is necessary for analyses herein.
	Multiplying the asymptotic variance of a sampler by the average time it takes to produce a sample point gives the computation time adjusted asymptotic variance.
	For the DG sampler, the adjusted asymptotic variance is
	\[
	\VD^{\dagger}(f) = \frac{(1+\tau)}{2} \VD(f).
	\]
	For the RG sampler, the adjusted asymptotic variance is
	\[
	\VR^{\dagger}(f,r) = (r\tau + 1-r) \VR(f,r).
	\]

	\subsection{Markov operators} \label{ssec:operator}

	Suppose that an Mtk~$K$ on $(\X,\B)$ has~$\pi$ as a stationary distribution in the sense that $\pi K = \pi$.
	Consider $L^2(\pi)$, the linear space of measurable functions $f: \X \to \mathbb{R}$ such that
	\[
	\|f\| := \sqrt{\pi f^2} < \infty.
	\]
	For $f,g \in L^2(\pi)$, say $f=g$ if $\|f-g\| = 0$, i.e., $f(x) = g(x)$ $\pi$-almost everywhere.
	Define an inner product
	\[
	\langle f, g \rangle = \int_{\X} f(x) g(x) \pi(\df x).
	\]
	Then $(L^2(\pi), \langle \cdot, \cdot \rangle)$ forms a real Hilbert space and $\|\cdot\|$ is the $L^2$ norm.
	%	For any signed measure~$\mu$ on $(\X,\B)$ and $f \in L^2(\pi)$, $\mu f$ is defined to be $\int_{\X} f(x) \mu(\df x)$, provided that the integral is well-defined.
	%	In particular, if $\df \mu/\df \pi$ exists and is in $L^2(\pi)$,
	%	\[
	%	\mu f = \left\langle \frac{\df \mu}{\df \pi}, f \right \rangle.
	%	\]
	It is convenient to define $L_0^2(\pi)$, the subspace of $L^2(\pi)$ that is orthogonal to constant functions.
	In other words, $L_0^2(\pi)$ consists of $f \in L^2(\pi)$ such that 
	\[
	\langle f, 1 \rangle = \pi f = 0.
	\]
	We can define a linear operation $f \mapsto Kf$ for $f \in L_0^2(\pi)$ in the following way:
	\[
	Kf(x) = \int_{\X} K(x,\df x') f(x'), \quad x \in \X.
	\]
	%	In terms of the associated Markov chain $(\tilde{X}_t)$, $(Kf)(x)$ is the conditional expectation of $f(\tilde{X}_{t+1})$ given $\tilde{X}_t = x$.
	One can verify that $Kf \in L_0^2(\pi)$ whenever $f \in L_0^2(\pi)$.
	Moreover, by Cauchy-Schwarz, $\|Kf\| \leq \|f\|$.
	Therefore,~$K$ can be regarded as a bounded linear operator on $L_0^2(\pi)$.
	This is called a Markov operator.
	The operator associated with the mixture (product) of two Mtks is simply the mixture (product) of their operators.

	The Mtks~$P_1$ and~$P_2$, as defined in~\eqref{eq:P1P2}, satisfy $\pi P_1 = \pi P_2 = \pi$, and thus give rise to the following operators on $L_0^2(\pi)$:
	\begin{equation} \label{eq:P1P2-operator}
		\begin{aligned}
			P_1f(x_1,x_2) &= \int_{\X_2} f(x_1,x_2') \pi_1(\df x_2'\mid x_1), \\
			P_2f(x_1,x_2) &= \int_{\X_1} f(x_1',x_2) \pi_2(\df x_1'\mid x_2).
		\end{aligned}
	\end{equation}
	The Mtks of the DG and RG samplers, as given in~\eqref{eq:PD} and~\eqref{eq:PR}, can then be treated as linear operators on $L_0^2(\pi)$ as well.
	It is straightforward to verify that $P_1$ and $P_2$ are self-adjoint, i.e., for $i = \{1,2\}$ and $f,g \in L_0^2(\pi)$,
	\[
	\langle P_i f, g \rangle = \langle f, P_i g \rangle.
	\]
	Moreover, $P_1$ and $P_2$ are idempotent, i.e., $P_1^2 = P_1$ and $P_2^2 = P_2$.
	Indeed, for $i \in \{1,2\}$, $P_i$ is the orthogonal projection onto the space of functions in $L_0^2(\pi)$ that only depend on $x_i$.
	
	%	Then the $t$-step Mtk of the DG algorithm, $\PD_t$, defines a linear operator that is an iterative product of two projections;
	%	the $t$-step Mtk of the RG algorithm, $\PR_t$, defines a linear operator that is a mixture of two projections raised to the $t$th power.

	Various quantities of interest can be studied within the above operator theoretic framework.
	For instance, if $K(x,\cdot)$ is the conditional distribution of some random element~$Y$ given $X=x$, where $X \sim \pi$, then, for any $f,g \in L_0^2(\pi)$,
	\[
	\langle f, Kg \rangle = \mbox{cov}(f(X), g(Y)).
	\]
	In particular, for $f \in L_0^2(\pi)$, the asymptotic variances in~\eqref{eq:VD-1} and~\eqref{eq:VR-1} can be written as
	\begin{equation} \label{eq:VD-innerproduct}
		\begin{aligned}
			\VD(f) = \|f\|^2 &+ \sum_{s=1}^{\infty} \langle f, \PD_{0,2s-1} f \rangle + \sum_{s=1}^{\infty} \langle f, \PD_{0,2s} f \rangle \\
			&+ \sum_{s=1}^{\infty} \langle f, \PD_{1,2s} f \rangle + \sum_{s=1}^{\infty} \langle f, \PD_{1,2s+1} f \rangle\\
			= \|f\|^2 &+ \sum_{s=1}^{\infty} \langle f, (P_1P_2)^{s-1} P_1 f \rangle + \sum_{s=1}^{\infty} \langle f, (P_1P_2)^s f \rangle \\
			&+ \sum_{s=1}^{\infty} \langle f, (P_2P_1)^{s-1}P_2 f \rangle + \sum_{s=1}^{\infty} \langle f, (P_2P_1)^s f \rangle,
		\end{aligned}
	\end{equation}
	and
	\begin{equation} \label{eq:VR-innerproduct}
		\begin{aligned}
			\VR(f,r) = \|f\|^2 + 2 \sum_{t=1}^{\infty} \langle f, \PR_t f \rangle = \|f\|^2 + 2 \sum_{t=1}^{\infty} \langle f, [(1-r) P_1 + rP_2]^t f \rangle.
		\end{aligned}
	\end{equation}
	These asymptotic variances will be carefully compared in Section~\ref{sec:main}.

	To continue, I will describe the algebraic foundation of this work.
	%	The main results of this paper will be stated afterwards.

	\subsection{Halmos's Theory of Two Projections} \label{ssec:twoproj}
	
	Just for this subsection, let $P_1$ and $P_2$ be two orthogonal projection operators on a generic Hilbert space $(H, \langle \cdot, \cdot \rangle)$.
	Let $H_1$ be the range of $P_1$, and $H_2$, that of $P_2$.
	Then~$H$ has the following orthogonal decomposition:
	\[
	H = M_{00} \oplus M_{01} \oplus M_{10} \oplus M_{11} \oplus \MRspace,
	\]
	where $\oplus$ denotes direct sum, $M_{00} = H_1 \cap H_2$, $M_{01} = H_1 \cap H_2^{\bot}$, $M_{10} = H_1^{\bot} \cap H_2$, $M_{11} = H_1^{\bot} \cap H_2^{\bot}$, and~$\MRspace$ is the rest.
	Here, $\bot$ denotes the orthogonal complement of a subspace.
	$P_1$ and $P_2$ leave each of $M_{00}$, $M_{01}$, $M_{10}$, $M_{11}$, and~$\MRspace$ invariant.
	Thus, $\MRspace$ can be further decomposed into $\MRspace = (P_1 \MRspace) \oplus (I-P_1)\MRspace$, where~$I$ is the identity operator on $\MRspace$.
	The behavior of $P_1$ and $P_2$ on $M_{ij}$ for $i,j \in \{0,1\}$ is simple.
	How the two projections act on $\MRspace = P_1\MRspace \oplus (I-P_1)\MRspace$ is, on the other hand, non-trivial.

	I will make use of an important result in linear algebra by \cite{halmos1969two}.
	Some additional concepts are needed to state the result.
	Let $G'$ and $H'$ be subspaces of some linear space, and let $G''$ and $H''$ be subspaces of another.
	Suppose that $G' \cap H' = 0$ and $G'' \cap H'' = 0$.
	A linear transformation~$A$ from $G' \oplus H'$ to $G'' \oplus H''$ has a $2 \times 2$ matrix representation.
	To see this, let~$P$ be the linear operator that maps any $g_0 + g_1 \in G'' \oplus H''$ (where $g_0 \in G''$ and $g_1 \in H''$) to $g_0 + 0$.
	Let $A_{00} = PA|_{G'}$ (where $A|_{G'}$ means~$A$ restricted to $G'$), $A_{01} = PA|_{H'}$, $A_{10} = (I-P)A|_{G'}$, $A_{11} = (I-P)A|_{H'}$. 
	Then, for $g_{0} + g_{1} \in G' \oplus H'$,
	\[
	A \sum_{i=0}^1 g_{i} = \sum_{i=0}^1 \sum_{j=0}^1 A_{ij} g_{j},
	\]
	and one may write
	\[
	A = \left( \begin{array}{cc}
		A_{00} & A_{01} \\
		A_{10} & A_{11}
	\end{array} \right).
	\]
	Linear combinations and products of linear transformations can be expressed as those of their matrix representations.
	Finally, a self-adjoint operator~$B$ on a generic Hilbert space is called a positive contraction if both~$B$ and identity minus~$B$ are positive semi-definite.
	
	Consider the matrix representations of $P_1$ and $P_2$ restricted to $\MRspace = P_1\MRspace \oplus (I-P_1)\MRspace$.
	They are characterized by the following famous result.

	\begin{lemma} \citep{halmos1969two} \label{lem:halmos}
		Assume that $\MRspace \neq \{0\}$.
		Then $P_1\MRspace$ and $(I-P_1)\MRspace$ are both nontrivial and have the same dimension.
		Moreover, there exist a unitary transformation $W: (I-P_1)\MRspace \to P_1 \MRspace$ and positive contractions~$C$ and~$S$ on $P_1 \MRspace$ that have the following properties:
		\begin{enumerate}
			\item $C^2 + S^2 = I_0$, where $I_0$ the identity operator on $P_1 \MRspace$;
			\item $\mbox{Ker}(C) = \mbox{Ker}(S) = \{0\}$, where $\mbox{Ker}$ denotes the kernel of an operator;
			\item
			\[
			\begin{aligned}
				P_1|_{\MRspace} &= 
				\Gamma^* \left( \begin{array}{cc}
					I_0 & 0 \\
					0 & 0
				\end{array} \right) \Gamma, \\
				P_2|_{\MRspace} &= 
				\Gamma^* \left( \begin{array}{cc}
					C^2 & CS \\
					CS & S^2
				\end{array} \right) \Gamma,
			\end{aligned}
			\]
			where
			\[
			\Gamma =  I_0 \oplus W := \left( \begin{array}{cc}
				I_0 & 0 \\
				0 & W
			\end{array} \right),
			\]
			and $\Gamma^* = I_0 \oplus W^*$ is its adjoint.
		\end{enumerate}
		
	\end{lemma}
	
	\begin{remark}
		$S$ is but a compact way of writing $\sqrt{I_0-C^2}$.
		In particular, $CS = SC$.
	\end{remark}
	
	\begin{remark}
		For more background on the theory of two projections, see \cite{bottcher2010gentle}.
		For a proof of Halmos's result, see \cite{bottcher2018robert}.
	\end{remark}
	
	Using Lemma~\ref{lem:halmos}, one can obtain useful representations of mixtures of two projections.
	For $r \in (0,1)$, 
	\[
	(1-r) P_1|_{\MRspace} + rP_2|_{\MRspace} = \Gamma^* \left( \begin{array}{cc}
		(1-r)I_0 + rC^2 & rCS \\
		rCS & rS^2
	\end{array} \right) \Gamma.
	\]
	Let
	\[
	\Delta(C) = (1-2r)^2 I_0 + 4r(1-r) C^2.
	\]
	The following is a result given by \cite{nishio1985structure}.
	
	\begin{lemma} \citep{nishio1985structure} \label{lem:structure}
		Assume that $\MRspace \neq \{0\}$.
		For $r \in (0,1)$, there exists a unitary operator $U: \MRspace \to \MRspace$ such that
		\[
		(1-r) P_1|_{\MRspace} + rP_2|_{\MRspace} = \Gamma^* U \left( \begin{array}{cc}
			[I_0+\sqrt{\Delta(C)}]/2 & 0 \\
			0 & [I_0-\sqrt{\Delta(C)}]/2
		\end{array} \right) U^* \Gamma.
		\]
	\end{lemma}
	
	An elementary proof of Lemma~\ref{lem:structure} is provided in the Supplement, along with a matrix representation of~$U$ which is not given in \cite{nishio1985structure}.

	In the context of Gibbs samplers, $H = L_0^2(\pi)$, and $P_1$ and $P_2$ are projection operators given by~\eqref{eq:P1P2-operator}.
	$H_1$ consists of functions $f \in L_0^2(\pi)$ such that $f(x_1,x_2)$ depends only on $x_1$, while
	$H_2$ consists of functions $f \in L_0^2(\pi)$ such that $f(x_1,x_2)$ depends only on $x_2$.
	The DG algorithm can be viewed as an alternating projection algorithm \citep{diaconis2010stochastic}.
	
	The subspaces $M_{ij}, \, i,j \in \{0,1\},$ also have interpretations, although they are not always easy to identify in practice.
	The space $M_{00}$ consists of $f \in L_0^2(\pi)$ such that $f(x_1,x_2)$ can be written as a function of just $x_1 \in \X_1$ as well as one of just $x_2 \in \X_2$.
	The space $M_{01}$ consists of $f \in L_0^2(\pi)$ such that $f(x_1,x_2)$ depends only on $x_1$ and that $E[f(X_1, X_2) | X_2] = E[f(X_1, X_2)] = 0$, where $(X_1, X_2) \sim \pi$.
	The space $M_{10}$ consists of $f \in L_0^2(\pi)$ such that $f(x_1,x_2)$ depends only on $x_2$ and that $E[f(X_1, X_2) | X_1] = 0$.
	The space $M_{11}$ consists of $f \in L_0^2(\pi)$ such that $E[f(X_1, X_2) | X_1] = 0$ and $E[f(X_1, X_2) | X_2] = 0$.
	Note that, for $f \in L_0^2(\pi)$, to say that $E[f(X_1, X_2) | X_i] = 0$ is to say that $f(X_1, X_2)$ is uncorrelated with any square integrable function of $X_i$.
	
	Starting from the next section, I will assume that $M_{00} = \{0\}$.
	To see what this assumption entails, consider the contrapositive, that is, there exist a nonzero $L_0^2(\pi)$ function~$f$, some $g: \X_1 \to \mathbb{R}$, and some $h: \X_2 \to \mathbb{R}$ such that $f(x,y) = g(x) = h(y)$ for $\pi$-a.e. $(x,y)$.
	Then there exists measurable $B \subset \mathbb{R}$ such that $\pi[f^{-1}(B)] > 0$ and $\pi[f^{-1}(B^c)] > 0$.
	Moreover, $f^{-1}(B)$ and $g^{-1}(B) \times h^{-1}(B)$ differ only by a $\pi$-measure zero set. 
	The same goes for $f^{-1}(B^c)$ and $g^{-1}(B^c) \times h^{-1}(B^c)$.
	Then~$\pi$ has positive mass on $B' = g^{-1}(B) \times h^{-1}(B)$ and $B'' = g^{-1}(B^c) \times h^{-1}(B^c)$, and zero mass on everything else.
	This is a ``reducible" structure that would render two-component Gibbs algorithms ineffective, as any chain that starts in $B'$ cannot enter $B''$, and any chain that starts in $B''$ cannot enter $B'$.
	
	Another assumption that will be made is that $\MRspace \neq \{0\}$.
	When $M_{00} = \{0\}$, this assumption states that $H_1$ and $H_2$ are not orthogonal.
	In the context of Gibbs samplers, this assumption is equivalent to $X_1$ not being independent of $X_2$, where $(X_1,X_2)$ has joint distribution~$\pi$.
	Indeed, for $f,g \in L_0^2(\pi)$,
	\[
	\langle f, g \rangle = \mbox{cov}[f(X_1,X_2), g(X_1,X_2)].
	\]
	Therefore, there exist $f \in H_1$ and $g \in H_2$ such that $\langle f,g \rangle \neq 0$ if and only if $X_1$ and $X_2$ are not independent.
	
	Finally, consider $\|C\|$, where $\|\cdot\|$ denotes the $L^2$ operator norm.
	(There is a slight abuse of notation since $\|\cdot\|$ is also used to denote the $L^2$ norm on $H = L_0^2(\pi)$.)
	It is well-known that $\|C\|$ is the cosine of the Friedrichs angle between $H_1$ and $H_2$, given by
	\[
	\sup \left\{ |\langle f, g \rangle|: \; f \in H_1 \cap (H_1 \cap H_2)^{\bot}, \; g \in H_2 \cap (H_1 \cap H_2)^{\bot}, \; \|f\|=\|g\|=1 \right\}.
	\]
	See, e.g., \cite{bottcher2010gentle}, Example 3.10.
	For $(X_1,X_2) \sim \pi$, this is the maximal correlation between~$X_1$ and~$X_2$.
	By Theorem 3.2 in \cite{liu:wong:kong:1994} along with Lemma 3.2, Theorem 4.1, and Proposition 3.5 in \cite{qin2021convergence}, when~$\B$ is countably generated, the DG/RG Markov chain is geometrically ergodic only if $\|C\| < 1$.

	\begin{example} \label{ex:normal0}
		Suppose that $\X_1 = \mathbb{R}^p$ for some positive integer~$p$, $\X_2 = \mathbb{R}$, and $\pi$ is the distribution of a $(p+1)$-dimensional normal vector 
		\[
		\left( \begin{array}{c}
			X_1 \\ X_2
		\end{array} \right) \sim \mbox{N}_{p+1} \left( \left( \begin{array}{c}
			m_1 \\ m_2
		\end{array} \right), \left( \begin{array}{cc}
			A & b \\
			b^{\top} & 1
		\end{array} \right)^{-1} \right),
		\]
		where $m_1 \in \mathbb{R}$, $m_2 \in \mathbb{R}^p$, $A \in \mathbb{R}^{p \times p}$ is positive definite, and $b \in \mathbb{R}^p$.
		To ensure that the precision matrix is positive definite, assume that $b^{\top}A^{-1}b < 1$.
		Assume also that $b \neq 0$, so that $X_1$ and $X_2$ are dependent.
		MCMC is not needed to sample from~$\pi$, but pretend that we are to apply two-component Gibbs anyway. 
		% use two-component Gibbs to estimate $\pi g$ for, say, $g(x_1,x_2) = x_2 +a^{-1}b^{\top} x_1$.
		The full conditional $\pi_1(\cdot|x_1)$ is a normal distribution with mean $m_2 - b^{\top}(x_1-m_1)$ and variance $1$, and $\pi_2(\cdot|x_2)$ is a $p$-variate normal distribution with mean $m_1-A^{-1}b(x_2-m_2)$ and covariance matrix $A^{-1}$.
		Evidently, $M_{00} = \{0\}$.
		The subspace $\MRspace$ may be difficult to conceptualize, but the following calculation gives two non-trivial elements of~$\MRspace$.
		Let $f \in L_0^2(\pi)$ be such that
		\begin{equation} \label{eq:normal-f}
			f(x_1,x_2) = x_2 - m_2 + b^{\top} (x_1-m_1).
		\end{equation}
		Claim: $f \in (I-P_1)\MRspace$.
		To see this, consider $P_1P_2 f$, given by
		\[
		P_1P_2f (x_1,x_2) = -(1-b^{\top}A^{-1}b) b^{\top}(x_1-m_1).
		\]
		Then $P_1P_2f \in H_1= M_{01} \oplus P_1\MRspace$.
		For any $g \in M_{01} = H_1 \cap H_2^{\bot}$, $\langle P_1P_2f, g \rangle = \langle P_2f, g \rangle = 0$.
		Thus $P_1P_2 f \in P_1\MRspace$.
		It's easy to verify that~$f$ is a linear combination of $P_1P_2f$ and $P_2P_1P_2f$, and $P_1 f = 0$.
		Since $\MRspace$ is invariant under $P_2$, it must hold that $f \in (I-P_1) \MRspace$.
		Finally, standard convergence analysis shows that the associated DG Markov chain is geometrically ergodic,  so $\|C\| < 1$.
		In fact, the squared maximal correlation $\|C\|^2$ between $X_1$ and $X_2$, is known to be $b^{\top} A^{-1} b$ \citep[see, e.g.,][Lemma~3]{amit1991rates}.
	\end{example}

	\section{Comparing Asymptotic Variances} \label{sec:main}
	
	We now use matrix representations of $P_1$ and $P_2$ to study the DG and RG algorithms.
	Throughout this section, assume that $M_{00} = \{0\}$, $\MRspace \neq \{0\}$, and that $\|C\| < 1$.
	
	\subsection{Characterizing $\VD(f)$ and $\VR(f,r)$}

%	In accordance with Lemma~\ref{lem:halmos}, we can write any linear operator on $L_0^2(\pi)$ that is in an algebra generated by $P_1$ and $P_2$ in matrix form.
	
	Fix $f \in L_0^2(\pi)$.
	Then one has the orthogonal decomposition
	\begin{equation} \label{eq:fdecomp}
		f = f_{00} + f_{01} + f_{10} + f_{11} + f_{0} + f_{1},
	\end{equation}
	where $f_{ij} \in M_{ij}$, $f_{0} \in P_1 \MRspace$, and $f_{1} \in (I-P_1)\MRspace$.
	%	As in Section~\ref{ssec:twoproj}, decompose~$f$ into
	%	\[
	%	f_{00} + f_{01} + f_{10} + f_{11} + f_{0} + f_{1}.
	%	\]
	Note that $f_{00} = 0$.
	Moreover, $P_1 f_{01} = f_{01}$, $P_1 f_{10} = P_1 f_{11} = 0$, $P_2 f_{10} = f_{10}$, and $P_2 f_{01} = P_2 f_{11} = 0$.

	Let us characterize the asymptotic variance of the DG sampler using Lemma~\ref{lem:halmos}.
	
	\begin{proposition} \label{pro:VD}
		\begin{equation} \label{eq:VD-2}
			\begin{aligned}
				\VD(f) &= 2\|f_{01}\|^2 + 2\|f_{10}\|^2 + \|f_{11}\|^2 + \langle f_{0} + f_{1}, \SD (f_{0} + f_{1}) \rangle,
				%\|f\|^2 + \|f_{01}\|^2 + \|f_{10}\|^2 + \langle f_{0} + f_{1}, \SD  (f_{0} + f_{1}) \rangle,
			\end{aligned}
		\end{equation}
		where
		\begin{equation} \label{eq:SD}
			\SD = \Gamma^* \left( \begin{array}{cc}
				2I_0 + 4C^2S^{-2} & 2CS^{-1} \\
				2CS^{-1} & 2I_0
			\end{array} \right)  \Gamma.
		\end{equation}
	\end{proposition}
	\begin{proof}
	Recall from~\eqref{eq:VD-innerproduct} that
	\[
	\begin{aligned}
		\VD(f) = \|f\|^2 &+ \sum_{s=1}^{\infty} \langle f, (P_1P_2)^{s-1} P_1 f \rangle + \sum_{s=1}^{\infty} \langle f, (P_1P_2)^s f \rangle \\
		&+ \sum_{s=1}^{\infty} \langle f, (P_2P_1)^{s-1}P_2 f \rangle + \sum_{s=1}^{\infty} \langle f, (P_2P_1)^s f \rangle.
	\end{aligned}
	\]
	By the decomposition of~$f$, this can be written as
	\begin{equation} \label{eq:VD-3}
		\begin{aligned}
			\VD(f) &= 2\|f_{01}\|^2 + 2\|f_{10}\|^2 + \|f_{11}\|^2 + \langle f_{0} + f_{1}, \Sigma (f_{0} + f_{1}) \rangle,
		\end{aligned}
	\end{equation}
	where
	\[
	\begin{aligned}
		\Sigma = I + & \sum_{s=1}^{\infty} (P_1|_{\MRspace} P_2|_{\MRspace})^{s-1} P_1|_{\MRspace} + \sum_{s=1}^{\infty} (P_1|_{\MRspace} P_2|_{\MRspace})^s + \\
		& \sum_{s=1}^{\infty}  (P_2|_{\MRspace} P_1|_{\MRspace})^{s-1} P_2|_{\MRspace} + \sum_{s=1}^{\infty} (P_2|_{\MRspace} P_1|_{\MRspace})^s.
	\end{aligned}
	\]
%	Write $P_1|_{\MRspace}$ and $P_2|_{\MRspace}$ into matrix forms as in Lemma~\ref{lem:halmos}. 
	For $s \geq 1$, we have the matrix representations
	\[
	(P_1|_{\MRspace} P_2|_{\MRspace})^s = \Gamma^* \left( \begin{array}{cc}
		C^{2s} & C^{2s-1}S \\
		0 & 0
	\end{array} \right) \Gamma, \qquad (P_2|_{\MRspace} P_1|_{\MRspace})^s = \Gamma^* \left( \begin{array}{cc}
		C^{2s} & 0 \\
		C^{2s-1}S & 0
	\end{array} \right) \Gamma.
	\]
	It follows that
	\[
	\Sigma = \Gamma^* \left( \begin{array}{cc}
		2I_0 + 4\sum_{s=1}^{\infty} C^{2s} & 2CS \sum_{s=0}^{\infty} C^{2s} \\
		2CS \sum_{s=0}^{\infty} C^{2s} & I_0 + S^2 \sum_{s=0}^{\infty} C^{2s}
	\end{array} \right)  \Gamma.
	\]
	Since $\|C\| < 1$,
	\[
	\sum_{s=0}^{\infty} C^{2s} = (I_0-C^2)^{-1} = S^{-2},
	\]
	and $\Sigma = \SD$.
	\eqref{eq:VD-3} then gives the desired result.
	\end{proof}

	Next, consider the RG sampler with selection probability $r \in (0,1)$.
	\begin{proposition} \label{pro:VR}
		\begin{equation} \label{eq:VR-2}
				\VR(f,r) = \frac{2-r}{r} \|f_{01}\|^2 + \frac{1+r}{1-r} \|f_{10}\|^2 + \|f_{11}\|^2 + \langle f_{0} + f_{1}, \SR(r) (f_{0} + f_{1}) \rangle,
		\end{equation}
		where
		\begin{equation} \label{eq:SR}
%			\begin{aligned}
				\SR(r) 
%				&= I + \frac{2}{r(1-r)} \Gamma^* \left( \begin{array}{cc}
%					(1-r)^2 I + C^2S^{-2} & rCS^{-1} \\
%					rCS^{-1} & r^2 I_0
%				\end{array} \right) \Gamma \\
%			&
			= \Gamma^* \left( \begin{array}{cc}
				\frac{2-r}{r} I_0 + \frac{2}{r(1-r)} C^2S^{-2} & \frac{2}{1-r} CS^{-1} \\
				\frac{2}{1-r} CS^{-1} & \frac{1+r}{1-r} I_0
			\end{array} \right) \Gamma.
%			\end{aligned}
		\end{equation}
	\end{proposition}
	\begin{proof}
	Note that
	\[
	[(1-r)P_1 + rP_2] f = (1-r) f_{01} + r f_{10} + [(1-r) P_1|_{\MRspace} + r P_2|_{\MRspace}] (f_{0} + f_{1}),
	\]
	where
	\[
	(1-r) P_1|_{\MRspace} + r P_2|_{\MRspace} = \Gamma^* \left( \begin{array}{cc}
		(1-r) I_0 + rC^2 & rCS \\
		rCS & rS^2
	\end{array} \right) \Gamma.
	\]
	By~\eqref{eq:VR-innerproduct}, the asymptotic variance is then
	\begin{equation} \label{eq:VR-3}
		\begin{aligned}
			\VR(f,r) &= \|f\|^2 + 2 \sum_{t=1}^{\infty} \langle f, [(1-r)P_1 + r P_2]^t f \rangle \\
			&= \|f\|^2 + 2 \sum_{t=1}^{\infty} \left[ (1-r)^t \|f_{01}\|^2 + r^t \|f_{10}\|^2 + \langle f_{0} + f_{1}, [(1-r)P_1|_{\MRspace} + rP_2|_{\MRspace}]^t (f_{0} + f_{1}) \rangle  \right] \\
			&= \frac{2-r}{r} \|f_{01}\|^2 + \frac{1+r}{1-r} \|f_{10}\|^2 + \|f_{11}\|^2 + \langle f_{0} + f_{1}, \Sigma (f_{0} + f_{1}) \rangle,
			%			&= \|f\|^2 + \frac{2(1-r)}{r} \|f_{01}\|^2 + \frac{2r}{1-r} \|f_{10}\|^2 + \langle f_{0} + f_{1}, \SR(r) (f_{0} + f_{1}) \rangle,
		\end{aligned}
	\end{equation}
	where
	\[
	\Sigma = I + 2\sum_{t=1}^{\infty} [(1-r) P_1|_{\MRspace} + rP_2|_{\MRspace}]^t.
	\]
	By Lemma~\ref{lem:structure},
	\[
	\|(1-r) P_1|_{\MRspace} + rP_2|_{\MRspace}\| = \frac{1}{2} \left\| I_0 + \sqrt{\Delta(C)} \right\|.
	\]
	It follows that the Neumann series in the expression of $\Sigma$ is convergent whenever $\|I_0 + \Delta(C)\| < 2$, which is equivalent to $\|C\| < 1$.
	$\Sigma$ can then be written as
	\[
	\begin{aligned}
		\Sigma &= 2 \left[  I - (1-r) P_1|_{\MRspace} - rP_2|_{\MRspace} \right]^{-1} - I \\
		&= 2 \Gamma^*  \left( \begin{array}{cc}
			rS^2 & -rCS \\
			-rCS & I_0 - rS^2
		\end{array} \right)^{-1}  \Gamma - I.
	\end{aligned}
	\]
	Evaluating the inverse may seem daunting, but it is completely analogous to calculating the inverse of a $2 \times 2$ matrix.
	It is straightforward to verify that for $x \in [0,1]$,
	\[
	\begin{aligned}
		&\left( \begin{array}{cc}
			r(1-x^2) & -rx\sqrt{1-x^2} \\
			-rx\sqrt{1-x^2} & 1 - r(1-x^2)
		\end{array} \right)^{-1} \\
		&= 
		\frac{1}{r(1-r)} \left( \begin{array}{cc}
			-r + (1-x^2)^{-1} & rx /\sqrt{1-x^2} \\
			rx /\sqrt{1-x^2} & r
		\end{array} \right).
	\end{aligned}
	\]
	Analogously,
	\[
	\left( \begin{array}{cc}
		rS^2 & -rCS \\
		-rCS & I_0	 - rS^2
	\end{array} \right)^{-1} = 
	\frac{1}{r(1-r)} \left( \begin{array}{cc}
		-rI_0 + S^{-2} & rCS^{-1} \\
		rCS^{-1} & rI_0
	\end{array} \right).
	\]
	Routine calculations show that $\Sigma = \SR(r)$, and the desired result follows from~\eqref{eq:VR-3}.
	\end{proof}

	\begin{example} \label{ex:normal1}
		This is a continuation of Example~\ref{ex:normal0}.
		Recall that~$f$, as given in~\eqref{eq:normal-f}, is in $(I-P_1)\MRspace$.
		By Propositions~\ref{pro:VD} and~\ref{pro:VR}, $\VD(f) = 2\|f\|^2$, while $\VR(f,r) = (1+r)\|f\|^2/(1-r)$.

	\end{example}

	\subsection{The comparison} \label{ssec:comparison}

	Propositions~\ref{pro:VD} and~\ref{pro:VR} allow one to compare $\VD^{\dagger}(f)$ and $\VR^{\dagger}(f,r)$ (as defined at the end of Section~\ref{ssec:basic}) for $f \in L_0^2(\pi)$, the asymptotic variances of the two samplers after adjusting for computation time.

	Consider first the unadjusted asymptotic variances.
	The key result of this section is as follows.
	
	\begin{proposition} \label{pro:VDVR}
		For $r \in (0,1)$ and $f \in L_0^2(\pi)$,
		\[
		\begin{aligned}
			\VD(f) &\leq k_1(r) \VR(f,r), \\
			\VR(f,r) &\leq k_2(r) \VD(f,r),
		\end{aligned}
		\]
		where
		\[
		\begin{aligned}
			k_1(r) &= 1-r+r^2 + \sqrt{r^2(1-r)^2 + (1-2r)^2}, \\
			k_2(r) &= \frac{1-r+r^2}{2r(1-r)} + \frac{\sqrt{r^2(1-r)^2 + (1-2r)^2}}{2r(1-r)}.
		\end{aligned}
		\]
	\end{proposition}
	
	\begin{proof}
		Decompose $f \in L_0^2(\pi)$ as in~\eqref{eq:fdecomp}.
		Recall~\eqref{eq:VD-2} and~\eqref{eq:VR-2}:
		\[
		\begin{aligned}
			\VD(f) &= 2\|f_{01}\|^2 + 2\|f_{10}\|^2 + \|f_{11}\|^2 + \langle f_{0} + f_{1}, \SD (f_{0} + f_{1}) \rangle,\\
			\VR(f,r) &= \frac{2-r}{r} \|f_{01}\|^2 + \frac{1+r}{1-r} \|f_{10}\|^2 + \|f_{11}\|^2 + \langle f_{0} + f_{1}, \SR(r) (f_{0} + f_{1}) \rangle.
		\end{aligned}
		\]
		By Lemma~\ref{lem:SDSR} in the Appendix, $\SD \leq k_1(r) \SR(r)$, and $\SR(r) \leq k_2(r) \SD$, where $\leq$ denotes Loewner ordering.
		Then
		\[
		\begin{aligned}
			\VD(f) &\leq \max \left\{ \frac{2r}{2-r}, \frac{2(1-r)}{1+r}, 1, k_1(r) \right\} \VR(f,r), \\
			\VR(f,r) &\leq \max \left\{ \frac{2-r}{2r}, \frac{1+r}{2(1-r)}, 1, k_2(r) \right\} \VD(f).
		\end{aligned}
		\]
		The result then follows from Lemma~\ref{lem:k1k2}.
	\end{proof}
	
	\begin{figure}
		\begin{center}
			\includegraphics[width=0.5\linewidth]{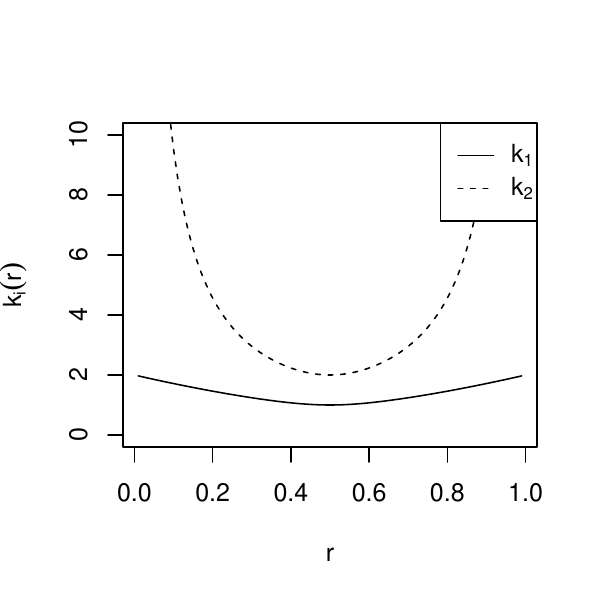}
			\caption{Graph of $r \mapsto k_i(r)$ for $i=1,2$.} \label{fig:kr}
		\end{center}
	\end{figure}
	
	The bounds in Proposition~\ref{pro:VDVR} are sharp in the sense that they cannot be improved without additional knowledge on~$\pi$ and~$f$.
	A more detailed explanation is given in Proposition~\ref{pro:bound-sharp}.
	It is easy to verify that $1 \leq k_1(r) \leq 2$, and $k_1(1/2)=1$;
	see Figure~\ref{fig:kr}.
	This means that when computation time is not taken into account, the DG algorithm is never much less efficient compared to the RG algorithm in terms of asymptotic variance.
	
	\begin{remark}
		Applying \eqref{eq:VD-2}-\eqref{eq:SR} for $r=1/2$ yields
		\[
		0 \leq \VR(f,1/2) - \|f\|^2 = 2[\VD(f) - \|f\|^2].
		\]
		This result can be found in \cite{greenwood1998information}.
	\end{remark}

	In practice, one needs to compare $\VD^{\dagger}(f) = (1+\tau)\VD(f)/2$ and $\VR^{\dagger}(f,r) = (r\tau + 1-r)\VR(f,r)$, where~$\tau$ is the time it takes to sample from $\pi_2(\cdot\mid x_2)$ if it takes unit time to sample from $\pi_1(\cdot\mid x_1)$.
	The following result follows immediately from Proposition~\ref{pro:VDVR}.
	
	\begin{corollary} \label{cor:VDstar}
		For $r \in (0,1)$ and $f \in L_0^2(\pi)$,
		\[
		\VD^{\dagger}(f) \leq \kappa(\tau,r) k_1(r) \VR^{\dagger}(f,r),
		\]
		where $k_1(r)$ is given in Proposition~\ref{pro:VDVR}, and
		\[
		\kappa(\tau,r) = \frac{\tau + 1}{2(r\tau +1  - r)}.
		\]
	\end{corollary}
	
	When $\tau \approx 1$, i.e., the times it takes to update $X_1$ and $X_2$ are roughly the same, $\kappa(\tau,r) \approx 1$, and $\VD^{\dagger}(f)$ is never much larger than $\VR^{\dagger}(f,r)$.
	On the other hand, if $\tau \ll 1$ or $\tau \gg 1$, i.e., the times it takes to update $X_1$ and $X_2$ are significantly different, then for some values of~$r$, e.g., $r=1/\sqrt{\tau}$ when $\tau \gg 1$ and $r = 1-\sqrt{\tau}$ when $\tau \ll 1$, $\kappa(\tau,r)$ can be very large.
	Intuitively, if $\tau \gg 1$, it makes sense to sample from $\pi_2(\cdot \mid x_2)$ less frequently, i.e., pick a smaller~$r$.
	The following toy example shows that when $\tau \gg 1$ and $r \ll 1$, $\VR^{\dagger}(f,r)$ can be considerably smaller than $\VD^{\dagger}(f)$ for some function~$f$.
	
	\begin{example} \label{ex:normal2}
		This is a continuation of Example~\ref{ex:normal1}.
		Recall that for the particular choice of~$f$, $\VD(f) = 2\|f\|^2$ and $\VR(f,r) = (1+r)\|f\|^2/(1-r)$.
		Suppose that~$p$ is large, and it takes a long time to update $X_1$ compared to $X_2$.
		Then $\tau \gg 1$.
		Avoid updating $X_1$ frequently by setting $r \ll 1$.
		Then
		\[
		\frac{\VR^{\dagger}(f,r)}{\VD^{\dagger}(f)} = \frac{r\tau + 1 - r}{\tau + 1} \frac{1+r}{1-r} \ll 1.
		\]
	\end{example}

	A general statement is given by the following result.
	\begin{corollary} \label{cor:VDstar-large}
		Let~$r$ be a function of~$\tau$ such that $r \to 0$ if $\tau \to \infty$.
		Then, for any nonzero $f \in M_{10} \oplus M_{11} \oplus (I-P_1) \MRspace$, as $\tau \to \infty$,
		\[
		\frac{\VR^{\dagger}(f,r)}{\VD^{\dagger}(f)} \to 0.
		\]
		By symmetry, a similar result holds as $\tau \to 0$ and $r \to 1$.
	\end{corollary}
	\begin{proof}
%		Lemma~\ref{lem:halmos} ensures that $(I-P_1)\MRspace \neq \{0\}$ whenever $\MRspace \neq \{0\}$.
		Let $f = f_{10} + f_{11} + f_{1} \in M_{10} \oplus M_{11} \oplus (I-P_1) \MRspace$ be a nonzero function.
		By~\eqref{eq:VD-2}-\eqref{eq:SR},
		\[
		\begin{aligned}
			&\VD(f) = 2\|f_{10}\|^2 + \|f_{11}\|^2 + 2\|f_{1}\|^2 \geq \|f\|^2, \\
			&\VR(f,r) = \frac{1+r}{1-r} \|f_{10}\|^2 + \|f_{11}\|^2 + \frac{1+r}{1-r}\|f_{1}\|^2 \leq \frac{1+r}{1-r} \|f\|^2.
		\end{aligned}
		\]
		Then
		\[
		\frac{\VR^{\dagger}(f,r)}{\VD^{\dagger}(f)} \leq \frac{2( r\tau + 1 - r) }{\tau + 1} \frac{1+r}{1-r}.
		\]
		The result follows.
	\end{proof}
	
	\begin{remark}
		The space $M_{10} \oplus M_{11} \oplus (I-P_1) \MRspace$ consists of functions that are orthogonal to $M_{01} \oplus P_1 \MRspace$, which is the space of functions $f \in L_0^2(\pi)$ such that $f(x_1, x_2)$ depends only on $x_1$.
		In other words, $f \in M_{10} \oplus M_{11} \oplus (I-P_1) \MRspace$ if and only if $f(X_1, X_2)$ is uncorrelated with any $L^2$ function of $(X_1, X_2)$ that only depends on $X_1$, where $(X_1, X_2) \sim \pi$.
	\end{remark}
	
	The question remains whether an RG sampler with a well-chosen~$r$ performs well compared to the DG sampler for an arbitrary function~$f$.
	This is answered by the following.
	
	\begin{corollary} \label{cor:VRstar}
		For $r \in (0,1)$ and $f \in L_0^2(\pi)$,
		\begin{equation} \label{ine:VRstar}
			\VR^{\dagger}(f,r) \leq \frac{k_2(r)}{\kappa(\tau,r)} \VD^{\dagger}(f) = \frac{2(r\tau +1  - r)}{\tau + 1} \frac{1-r+r^2 + \sqrt{r^2(1-r)^2 + (1-2r)^2}}{2r(1-r)}    \VD^{\dagger}(f),
		\end{equation}
		where $k_2(r)$ is given in Proposition~\ref{pro:VDVR}, and $k(\tau,r)$ is given in Corollary~\ref{cor:VDstar}.
		%		In particular, when
		%		\[
		%		r = \frac{1}{1+\tau^{1/3}},
		%		\]
		%		$k_2(r)/\kappa(\tau,r) \leq 2$, and thus $\VR^{\dagger}(f,r) \leq 2 \VD(f)$.
		In particular, if we let $r=0.5$ when $\tau=1$, and
		\begin{equation} \label{eq:r}
			r = \frac{-2\tau - 1 + \sqrt{\tau(2\tau+1)(\tau+2)}}{\tau^2 - 1} \in (0,1)
		\end{equation}
		when $\tau \neq 1$, then $k_2(r)/\kappa(\tau,r)=2$, and thus
		\begin{equation} \nonumber
			\VR^{\dagger}(f,r) \leq 2 \VD^{\dagger}(f).
		\end{equation}
	\end{corollary}
	
	\begin{proof}
		\eqref{ine:VRstar} follows from Proposition~\ref{pro:VDVR}.
		
		When $\tau = 1$ and $r = 0.5$, the result clearly holds.
		Fix $\tau \neq 1$ and let~$r$ be as in~\eqref{eq:r}.
		It's straightforward to verify that $r \in (0,1)$.
		Moreover, solving for~$\tau$ yields
		\[
		\tau = \frac{1-2r + \sqrt{(1-2r)^2+r^2(1-r)^2}}{r^2}.
		\]
		Then
		\[
		\frac{1}{\kappa(\tau,r)} = \frac{2r[1-r-r^2 + \sqrt{(1-2r)^2+r^2(1-r)^2}]}{(1-r)^2 + \sqrt{(1-2r)^2+r^2(1-r)^2}}.
		\]
		Routine calculations show that $k_2(r)/\kappa(\tau,r) = 2$.
	\end{proof}
	
	\begin{figure}[h]
		\begin{center}
			\includegraphics[width=0.5\linewidth]{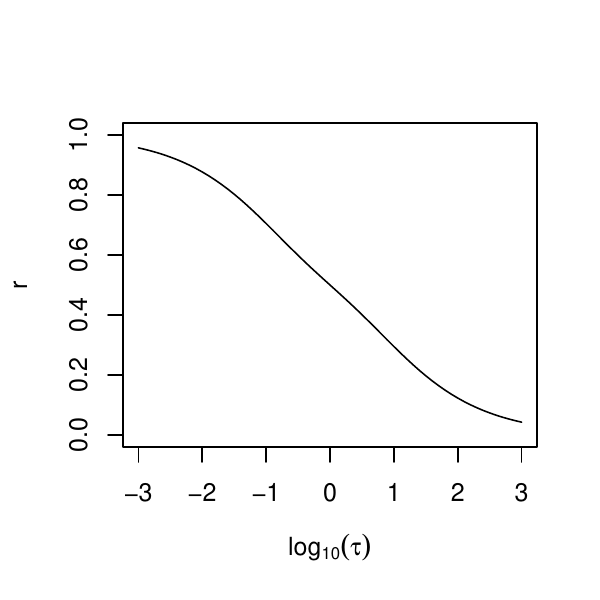}
			\caption{Relationship between $\tau$ and $r$ as given in~\eqref{eq:r}.
				$\tau$ is given in log scale.} \label{fig:rtau}
		\end{center}
	\end{figure}
	
	The relationship between~$r$ and~$\tau$ as given by~\eqref{eq:r} is plotted in Figure~\ref{fig:rtau}.
	It can be shown that this choice of~$r$ is not much worse than optimal in terms of minimizing $k_2(r)/\kappa(\tau,r)$.
	Moreover, as $\tau \to \infty$, $r \to 0$, and as $\tau \to 0$, $r \to 1$.
	In other words, this choice of~$r$ also satisfies the conditions in Corollary~\ref{cor:VDstar-large}.

	Corollaries~\ref{cor:VDstar-large} and~\ref{cor:VRstar} show that, in terms of adjusted asymptotic variance, when $\tau$ is extremely large or extremely small, the RG sampler with a good choice of~$r$ is never significantly worse than the DG sampler, and can be considerably better in some cases.
	One can say that the RG sampler is more robust in this regard.
	
	One may also use other choices of $r$.
	By Corollaries~\ref{cor:VDstar-large} and~\ref{cor:VRstar}, similar effects can be achieved whenever $r$ is a decreasing function of $\tau$ such that $1 \ll 1/r = O(\tau)$ when $\tau \to \infty$ and $1 \ll 1/(1-r) = O(1/\tau)$ when $\tau \to 0$.
	Here, $1/r = O(\tau)$ means $1/(r\tau)$ is bounded.
	
	\subsection{Data augmentation} \label{ssec:da}
	
	Studying the asymptotic variance $V(f)$ for all $f \in L_0^2(\pi)$ is relevant only if we care about integrating all $L^2$ functions, which is not always the case.
%	It must be noted that the comparison between $\VD^{\dagger}(f)$ and $\VR^{\dagger}(f,r)$ is relevant only if one is interested in estimating $\pi g$ for some function $g \in L^2(\pi)$ such that $f = g - \pi g$.
	For example, in a data augmentation setting \citep{tanner1987calculation,van2001art}, one usually only wishes to estimate $\pi f$ for $f \in L^2(\pi)$ such that $f(x_1,x_2)$ just depends on $x_1$.
	In this case, it is only interesting to compare $\VD^{\dagger}(f)$ and $\VR^{\dagger}(f,r)$ for $f \in H_1 = M_{01} \oplus P_1\MRspace$.
	%	
	%	\begin{remark}
		%		In a data augmentation setting, half of the data points produced by the DG sampler are redundant, as $X_1$ is updated once every two iterations.
		%		Thus, a practitioner would use a thinned version of the Markov chain, where only one sample point is recorded every two iterations.
		%		This procedure does not change the value or distribution of the sample mean for functions $g(x_1,x_2)$ that depend only on~$x_1$.
		%	\end{remark}
	
	Let $f = f_{01} + f_{0} \in M_{01} \oplus P_1\MRspace$.
	By~\eqref{eq:VD-2}-\eqref{eq:SR},
	\begin{equation} \label{eq:V-da}
		\begin{aligned}
			\VD(f) &= 2\|f_{01}\|^2 + \langle f_{0}, (2I_0 + 4C^2S^{-2}) f_{0} \rangle, \\
			\VR(f,r) &= \frac{2-r}{r} \|f_{01}\|^2 + \left\langle f_{0}, \left( \frac{2-r}{r} I_0 + \frac{2}{r(1-r)} C^2S^{-2} \right) f_{0} \right\rangle.
		\end{aligned}
	\end{equation}
	Then
	\[
	\VD(f) \leq \max\left\{ \frac{2r}{2-r}, 2r(1-r) \right\} \VR(f,r).
	%	, \quad \VR(f,r) \leq \max \left\{ \frac{2-r}{2r}, \frac{1}{2r(1-r)} \right\} \VD(f).
	\]
	Let $\kappa(\tau,r)$ be defined as in Corollary~\ref{cor:VDstar}.
	Then
	\[
	\VD^{\dagger}(f) \leq \kappa(\tau,r) \max\left\{ \frac{2r}{2-r}, 2r(1-r) \right\} \VR^{\dagger}(f,r).
	\]
	Now,
	\[
	\kappa(\tau,r) \max\left\{ \frac{2r}{2-r}, 2r(1-r) \right\} \leq \frac{\tau+1}{2r\tau}  \max\left\{ \frac{2r}{2-r}, 2r(1-r) \right\} \leq \frac{\tau+1}{\tau}.
	\]
	Thus, when~$\tau$ is not too small, the DG algorithm can compete with the RG algorithm in terms of adjusted asymptotic variance.
	When $\tau \ll 1$, depending on the structure of~$\pi$, a result like Corollary~\ref{cor:VDstar-large} could still hold for some $f \in H_1$.

	%	When $\tau \ll 1$, whether a result like Corollary~\ref{cor:VDstar-large} holds for some $f \in H_1$ depends on the structure of~$\pi$.
	%	For instance, suppose that there exists a nonzero function $f \in M_{01}$.
	%	That is, $f(x_1,x_2)$ depends only on~$x_1$, and the following holds almost everywhere:
	%	\[
	%	P_2f(x_1,x_2) = \int_{\X_1} f(x_1',x_2) \pi_2(\df x_1'|x_2) = 0.
	%	\]
	%	Then, by~\eqref{eq:V-da}, $\VD(f) = 2\|f\|^2$, $\VR(f,r) = (2-r)\|f\|^2/r$.
	%	It follows that
	%	\[
	%	\VR^{\dagger}(f,r) = \frac{2(r\tau+1-r)}{\tau + 1} \frac{2-r}{2r} \VD^{\dagger}(f).
	%	\]
	%	Letting~$r$ be a function of~$\tau$ that goes to~1 as $\tau \to 0$ will result in
	%	\[
	%	\frac{\VR^{\dagger}(f,r)}{\VD^{\dagger}(f)} \to 0.
	%	\]
	%	For functions outside $M_{01}$, the situation is more complicated and will not be discussed in further detail.
	%	The take-away message is that general results like Corollaries~\ref{cor:VDstar} and~\ref{cor:VRstar} can be refined depending on the class of functions~$f$ that one is interested in.
	
	\subsection{A Modified DG Sampler} \label{ssec:modified}
	
	The RG sampler achieves robustness by updating the less costly component more frequently.
	Naturally, one can use this idea to modify the DG sampler.
	Suppose that $\tau \geq 1$, i.e., it is more costly to sample from $\pi_2(\cdot \mid x_2)$ (which updates $X_1$) than $\pi_1(\cdot \mid x_1)$ (which updates $X_2$).
	Consider the following algorithm, which updates $X_2$ $\ell$ times consecutively before updating $X_1$ once, where $\ell \geq 1$.
	It simulates a Markov chain that again has~$\pi$ as its stationary distribution.
	
	\begin{algorithm}[H]
		\caption{Modified DG sampler with $\ell$ repeated draws from $\pi_1$} \label{alg:modified}
		Draw $\tilde{X}_0 = (X_{1,0},X_{2,0})$ from some initial distribution on $(\X, \B)$, and
		set $t=0$\;
		
		\While{$t < T$}{
			\If{$t = s(\ell+1) + q$ for some non-negative integer $s$ and $q \in \{0,\dots,\ell-1\}$ }{
				draw $X_{2,t+1}$ from $\pi_1(\cdot\mid X_{1,t})$, set $X_{1,t+1} = X_{1,t}$, and let $\tilde{X}_{t+1} = (X_{1,t+1},X_{2,t+1})$\;
			}
			\If{$t= s(\ell+1) + \ell$ for some non-negative integer~$s$}{
				draw $X_{1,t+1}$ from $\pi_2(\cdot\mid X_{2,t})$, set $X_{2,t+1} = X_{2,t}$, and let $\tilde{X}_{t+1} = (X_{1,t+1},X_{2,t+1})$\;
			}
			set $t=t+1$ \;
		}
	\end{algorithm}

	Let $f \in L_0^2(\pi)$. 
	Let $\VM(f,\ell)$ be the unadjusted asymptotic variance associated with this function and the modified DG algorithm.
	By Proposition~2 of \cite{greenwood1998information}, 
	\[
	\begin{aligned}
		\VM(f,\ell) = \|f\|^2 + & \frac{2}{\ell+1} \left[ \langle f, P_2 f \rangle + \sum_{s=1}^{\infty} \langle f, (P_2 P_1)^s P_2 f \rangle + \ell \sum_{s=1}^{\infty} \langle f, (P_2P_1)^s f \rangle \right. \\
		+ & \left. \frac{\ell(\ell+1)}{2} \langle f, P_1 f \rangle + \ell \sum_{s=1}^{\infty} \langle f, (P_1P_2)^s f \rangle + \ell^2 \sum_{s=1}^{\infty} \langle f, (P_1P_2)^s P_1 f \rangle \right].
	\end{aligned}
	\]
	
	Assume that $M_{00} = \{0\}$, $\MRspace \neq \{0\}$, and that $\|C\| < 1$. 
	% Decompose~$f$ as in~\eqref{eq:fdecomp}.
	Similarly to how Propositions~\ref{pro:VD} and~\ref{pro:VR} are proved, one can establish the following.
	\begin{proposition} \label{pro:VM}
		\[
		\VM(f,\ell) = (\ell+1) \|f_{01}\|^2 + \frac{\ell+3}{\ell+1} \|f_{10}\|^2 + \|f_{11}\|^2 + \langle f_{0} + f_{1}, \SM(\ell) (f_{0} + f_{1}) \rangle,
		\]
		where
		\begin{equation} \label{eq:SM}
		\SM(\ell) = \Gamma^* \left( \begin{array}{cc}
			(\ell+1) I_0 + 2(\ell+1) C^2 S^{-2} & 2 CS^{-1} \\
			2CS^{-1} & \frac{\ell+3}{\ell+1} I_0
		\end{array} \right) \Gamma.
		\end{equation}
	\end{proposition}
	
	\begin{remark}
		If instead one updates $X_2$ once before updating $X_1$ $\ell$ times, then the asymptotic variance is
		\[
		\VM(f,\ell^{-1}) = \frac{\ell+3}{\ell + 1} \|f_{01}\|^2 + (\ell+1) \|f_{10}\|^2 + \|f_{11}\|^2 + \left\langle f_{0} + f_{1}, \SM(\ell^{-1}) (f_{0} + f_{1}) \right\rangle,
		\]
		where
		\[
		\SM(\ell^{-1}) = \Gamma^* \left( \begin{array}{cc}
			\frac{\ell + 3}{\ell+1} I_0 + 2(\ell+1) C^2 S^{-2} - \frac{\ell(\ell-1)}{\ell+1} C^2 & 2 \ell CS^{-1} - \frac{\ell(\ell-1)}{\ell+1}CS \\
			2 \ell CS^{-1} - \frac{\ell(\ell-1)}{\ell+1}CS & \frac{2\ell^2+\ell+1}{\ell+1} I_0 - \frac{\ell(\ell-1)}{\ell+1} S^2
		\end{array} \right) \Gamma.
		\]
	\end{remark}
	
	To proceed, let us consider the computational cost of the modified DG algorithm.
	If the algorithm is run verbatim, then, on average, each iteration takes $(\tau+\ell)/(\ell+1)$ units of time.
	It is, however, possible to save time through a parallelization scheme.
	Note that, in Algorithm~\ref{alg:modified}, we may first simulate the subchain
	\[
	\tilde{X}_0, \tilde{X}_{\ell}, \tilde{X}_{\ell+1}, \tilde{X}_{(\ell+1)+\ell}, \tilde{X}_{2(\ell+1)}, \dots, \tilde{X}_{(s-1)(\ell+1) + \ell}, \tilde{X}_{s(\ell+1)}, \dots
	\]
	through the standard DG sampler (Algorithm~\ref{alg:dg}).
	After simulating the subchain at a given length, for $t = s (\ell+1) + q$ where $s$ is a non-negative integer and $q \in \{1,\dots,\ell-1\}$, $\tilde{X}_t = (X_{1,t}, X_{2,t})$ can be generated by setting $X_{1,t} = X_{1,s(\ell+1)}$, and drawing $X_{2,t}$ from $\pi_1(\cdot \mid X_{1,s(\ell+1)})$.
	Hence, the repeated sampling from $\pi_1$ can be done through post-processing.
	Moreover, in an ideal setting, this can be executed in parallel, resulting in the overall runtime of the modified DG algorithm being essentially equal to the time it takes to simulate the subchain.
	A modified DG chain of length $T$ corresponds to a subchain of length roughly $2T/(\ell+1)$.
	Thus, depending on how much parallelization is feasible, the per-iteration cost (in terms of time) of the modified DG algorithm can vary between $2/(\ell+1) \times (\tau+1)/2 = (\tau+1)/(\ell+1)$ and $(\tau+\ell)/(\ell+1)$.
	
	I will first analyze the modified DG sampler when assuming that there is no parallelization.
	It will be shown that even in the absence of parallelization, this sampler is competitive against the standard DG and RG samplers.
	
	The adjusted asymptotic variance for the modified DG algorithm is
	\[
	\VM^{\dagger}(f,\ell) = \frac{\ell+\tau}{\ell + 1} \VM(f,\ell).
	\]
	Let us now compare $\VM^{\dagger}(f,\ell)$ to $\VD^{\dagger}(f)$ and $\VR^{\dagger}(f,r)$.
	
	One would expect that the modified DG sampler behaves similarly to the RG sampler with selection probability $r = 1/(\ell+1)$.
	The following result shows that the modified DG sampler is better than, but comparable to the RG sampler with $r = 1/(\ell+1)$ in terms of adjusted asymptotic variance.
	\begin{corollary} \label{cor:VMVR-1}
		For $\ell \geq 1$ and $f \in L_0^2(\pi)$,
		\[
		\VM^{\dagger}(f,\ell) \leq \VR^{\dagger}\left(f, \frac{1}{\ell+1} \right) \leq \max \left\{ \frac{2\ell+1}{\ell+1}, \frac{\ell+1}{\ell}  \right\} \VM^{\dagger}(f,\ell).
		\]
	\end{corollary}
	\begin{proof}
		Fix $\ell \geq 1$ and $f \in L_0^2(\pi)$.
		By Lemma~\ref{lem:SM}, 
		\[
		\SM(\ell) \leq \SR \left( \frac{1}{\ell+1} \right) \leq \max \left\{ \frac{2\ell+1}{\ell+1}, \frac{\ell+1}{\ell}  \right\} \SM(\ell).
		\]
		It then follows from Propositions~\ref{pro:VR} and~\ref{pro:VM} that
		\[
		\VM(f,\ell) \leq \VR\left(f, \frac{1}{\ell+1} \right) \leq \max \left\{ \frac{2\ell+1}{\ell+1}, \frac{\ell+1}{\ell}  \right\} \VM(f,\ell).
		\]
		The desired result is obtained by noting that $1-r+r\tau = (\ell + \tau)/(\ell+1)$ if $r = 1/(\ell+1)$.
	\end{proof}
	
	A comparison between $\VR^{\dagger}(r)$ and $\VM^{\dagger}(\ell)$ for a general~$r$ that is not $1/(\ell+1)$ can be conducted on a case-by-case basis.
	The key step is verifying whether operators of the form $k \SM(\ell) - \SR(r)$ and $k \SR(r) - \SM(\ell)$ are positive semi-definite.
	A tool for this is provided in Lemma~\ref{lem:positive}.
	
	Naturally, just like the RG sampler, the modified DG sampler enjoys robustness over the vanilla DG sampler, given that~$\ell$ is well-chosen.
	The next two results are derived from Propositions~\ref{pro:VD},~\ref{pro:VM}, and Lemma~\ref{lem:SM}.
	The proofs are extremely similar to those of Corollaries~\ref{cor:VDstar-large} and~\ref{cor:VRstar}, so they will be omitted.
	
	\begin{corollary}
		Let~$\ell$ be a function of~$\tau$ such that $\ell \to \infty$ if $\tau \to \infty$.
		Then, for any nonzero $f \in M_{10} \oplus M_{11} \oplus (I-P_1) \MRspace$, as $\tau \to \infty$,
		\[
		\frac{\VM^{\dagger}(f,\ell)}{\VD^{\dagger}(f)} \to 0.
		\]
	\end{corollary}
	
	\begin{corollary} \label{cor:VMVR-3}
		For $\ell \geq 1$ and $f \in L_0^2(\pi)$,
		\[
		\VM^{\dagger}(f,\ell) \leq \frac{2(\ell+\tau)}{(\ell+1)(\tau+1)} \frac{\ell^2+\ell+2+(\ell-1)\sqrt{(\ell+1)^2+1}}{2(\ell+1)} \VD^{\dagger}(f).
		\]
		In particular, if
		\begin{equation} \label{ine:l}
		\ell \leq \frac{1+\sqrt{1+ 4(\tau + 1)}}{2},
		\end{equation}
		then
		\[
		\VM^{\dagger}(f,\ell) \leq 2 \VD^{\dagger}(f).
		\]
	\end{corollary}
	
	If $\tau \gg 1$, then letting $\ell \gg 1$ under the restriction of~\eqref{ine:l} would give the desired robustness.
	In fact, having $\ell \gg 1$ and $\ell = O(\tau)$ would yield a similar effect.
	
	When post-processing and parallelization is fully exploited, $\VM^{\dagger}(f,\ell)$ can be further diminished.
	Indeed, if we assume full parallelization and consider only the cost of computation time, then $\VM^{\dagger}(f, \ell) = (\tau+1) \VM(f,\ell)/(\ell+1)$, rather than $(\tau+\ell) \VM(f,\ell) / (\ell+1)$ as before.
	One may then multiply a factor of $(\tau+1)/(\tau+\ell)$ to the any upper or lower bounds on $\VM^{\dagger}(\ell)$ in Corollaries \ref{cor:VMVR-1} and \ref{cor:VMVR-3}.

	\section{Convergence Rate} \label{sec:rate}
	
	I will now use \pcite{halmos1969two} theory to study the convergence properties of two-component Gibbs samplers, and provide an alternative proof of \cite{qin2021convergence} main result, which, loosely speaking, states that the DG Markov chain converges faster than the RG chain.
%	Along the way, I will prove some general results regarding the convergence rates of Markov chains, which may be useful for analyzing other types of Markov chains.
%	What separates these results from others that have appeared in the MCMC
%	literature \citep[e.g.,][]{roberts1997geometric} is that I consider chains that are possibly non-reversible
%	and/or time-inhomogeneous.  
	
	I will first provide some general results regarding the convergence rates of possibly time-inhomogeneous Markov chains.
	Define $L_*^2(\pi)$ to be the set of probability measures~$\mu$ such that~$\mu$ is absolutely continuous with respect to~$\pi$, and that $\df \mu/ \df \pi \in L^2(\pi)$.
	For $\mu,\nu \in L_*^2(\pi)$, one can define their $L^2$ distance
	\[
	\|\mu - \nu\|_* = \sup_{f \in L_0^2(\pi), \, \|f\|=1} (\mu f - \nu f) = \left\| \frac{\df \mu}{\df \pi} - \frac{\df \nu}{\df \pi} \right\|.
	\]
	This distance is often used in the convergence analysis of Markov chains.
	Let $K_t$ be the $t$-step transition kernel of a possibly time-inhomogeneous Markov chain such that~$\pi$ is invariant, i.e., $\pi K_t = \pi$ for each positive integer~$t$.
	It can be checked that $\mu K_t \in L_*^2(\pi)$ whenever $\mu \in L_*^2(\pi)$.
	If $(\tilde{X}_t)$ is a chain associated with this Mtk and $\tilde{X}_0 \sim \mu \in L_*^2(\pi)$, then $\mu K_t$ is the marginal distribution of $\tilde{X}_t$. 
	The $L^2$ distance between $\mu K_t$ and~$\pi$ is
	\begin{equation} \label{eq:l2dist}
		\begin{aligned}
			\|\mu K_t - \pi\|_* &= \sup_{f \in L_0^2(\pi), \, \|f\|=1} (\mu-\pi) K_t f \\
			&= \sup_{f \in L_0^2(\pi), \, \|f\|=1} \left\langle \frac{\df \mu}{\df \pi} - 1, K_tf \right\rangle \\
			&= \left\| K_t^* \left( \frac{\df \mu}{\df \pi} - 1 \right)  \right\| \\
			&= \left\langle \frac{\df \mu}{\df \pi} - 1, K_tK_t^* \left( \frac{\df \mu}{\df \pi} - 1 \right) \right\rangle^{1/2},
		\end{aligned}
	\end{equation}
	where $K_t^*$ is the adjoint of $K_t$.
	
	Evidently, the behavior of the positive semi-definite operator $K_tK_t^*$ as $t \to \infty$ is closely related to the convergence properties of the chain.
	To make this relationship more explicit, consider the convergence rate
	\begin{equation} \nonumber
		\rho_0 = \exp \left( \sup_{\mu \in L_*^2(\pi),\; \mu \neq \pi} \limsup_{t \to \infty} t^{-1} \log \|\mu K_t - \pi\|_* \right)
	\end{equation}
	\citep[][]{roberts2001geometric,qin2020limitations}.
	It can be shown that $\rho_0 \in [0,1]$.
	If $\rho_0< \rho \in (0,1]$, then for each $\mu \in L_*^2(\pi)$ there exists $M_{\mu} < \infty$ such that
	\begin{equation} \label{ine:geometric}
		\|\mu K_t - \pi\|_* < M_{\mu} \rho^t
	\end{equation}
	for every positive integer~$t$.
	Conversely, if~\eqref{ine:geometric} holds for each~$t$, then $\rho_0 \leq \rho$.
	Hence, the smaller the rate, the faster the convergence.
	
	The following result is proved using standard techniques like those from \cite{fill1991eigenvalue,roberts1997geometric,paulin2015concentration}, with details given in Appendix~\ref{app:rate}.
	Readers are also referred to \cite{kontorovich2008concentration}, which establishes concentration inequalities for time-inhomogeneous chains.

	\begin{proposition} \label{pro:rate}
		Let $K_t$ be the $t$-step transition kernel of a Markov chain with stationary distribution~$\pi$.
		Then
		\[
		\sup_{\mu \in L_*^2(\pi),\; \mu \neq \pi} \limsup_{t \to \infty} t^{-1} \log \|\mu K_t - \pi\|_* = \log \left( \sup_{f \in L_0^2(\pi), \; \|f\|=1} \limsup_{t \to \infty} \langle f, K_t K_t^* f \rangle^{1/(2t)} \right).
		\]
		That is, 
		\[
		\rho_0 = \sup_{f \in L_0^2(\pi), \; \|f\|=1} \limsup_{t \to \infty} \langle f, K_t K_t^* f \rangle^{1/(2t)}.
		\]
	\end{proposition}

	Let $\rhoD$ and $\rhoR(r)$ be the convergence rates of the DG and RG chains, respectively, where~$r$ is the selection probability.
	That is,
	\[
	\begin{aligned}
		\rhoD := \exp \left( \sup_{\mu \in L_*^2(\pi),\; \mu \neq \pi} \limsup_{t \to \infty} t^{-1} \log \|\mu \PD_t - \pi\|_* \right), \\
		\rhoR(r) := \exp \left( \sup_{\mu \in L_*^2(\pi),\; \mu \neq \pi} \limsup_{t \to \infty} t^{-1} \log \|\mu R_t - \pi\|_* \right).
	\end{aligned}
	\]
	%	where $\PD_t$ and $\PR_t$ are, respectively, the $t$-step Mtks of the DG and RG chains.
	Let us now use Proposition~\ref{pro:rate} to find the relationship between the two rates.
	The following lemma is useful.
	
	\begin{lemma} \label{lem:At}
		Suppose that in Proposition~\ref{pro:rate}, $K_tK_t^* = A^{m(t)}$ for~$t$ sufficiently large, where $A$ is a positive contraction on $L_0^2(\pi)$ and $m(t)$ is a non-negative function of~$t$.
		Then
		\[
		\sup_{f \in L_0^2(\pi), \; \|f\|=1} \limsup_{t \to \infty} \langle f, K_t K_t^* f \rangle^{1/(2t)} = \|A\|^m,
		\]
		where
		\[
		m = \liminf_{t \to \infty} (2t)^{-1} m(t).
		\]
	\end{lemma}
	\begin{proof}
		Because $A$ is positive semi-definite and $m(t)$ is non-negative, for $f \in L_0^2(\pi)$,
		\[
		\langle f, A^{m(t)} f \rangle \leq \|A^{m(t)}\| = \|A\|^{m(t)}.
		\]
		Then, since $\|A\| \leq 1$,
		\[
		\sup_{f \in L_0^2(\pi), \; \|f\|=1} \limsup_{t \to \infty} \langle f, K_tK_t^* f \rangle^{1/(2t)} = \sup_{f \in L_0^2(\pi), \; \|f\|=1} \limsup_{t \to \infty} \langle f, A^{m(t)} f \rangle^{1/(2t)} \leq \|A\|^m. 
		\]
		Consider the reverse inequality,
		which obviously holds when $\|A\| = 0$.
		Suppose that $\|A\| > 0$, and let $\varepsilon \in (0,\|A\|)$ be arbitrary.
		$A$ has a spectral decomposition $A = \int_{-\infty}^{\infty} \lambda E_A(\df \lambda)$, where $E_A$ is the projection-valued measure associated with~$A$ \citep[see, e.g.,][Theorem 3.15]{kubrusly2012spectral}.
		There exists a function $f_{(\varepsilon)}$ in the range of $E_A((\|A\|-\varepsilon, \|A\|])$ such that $\|f_{(\varepsilon)}\| = 1$.
		Then
		\[
		\begin{aligned}
			\sup_{f \in L_0^2(\pi), \; \|f\|=1} \limsup_{t \to \infty} \langle f, K_tK_t^* f \rangle^{1/(2t)} & \geq  \limsup_{t \to \infty} \langle f_{(\varepsilon)}, A^{m(t)} f_{(\varepsilon)} \rangle^{1/(2t)} \\
			&=  \limsup_{t \to \infty} \left( \int_{-\infty}^{\infty} \lambda^{m(t)} \langle f_{(\varepsilon)}, E_A(\df \lambda) f_{(\varepsilon)} \rangle \right)^{1/(2t)} \\
			& \geq (\|A\|-\varepsilon)^m.
		\end{aligned}
		\]
		Since $\varepsilon$ is arbitrary,
		\[
		\sup_{f \in L_0^2(\pi), \; \|f\|=1} \limsup_{t \to \infty} \langle f, K_tK_t^* f \rangle^{1/(2t)} \geq \|A\|^m.
		\]
	\end{proof}
	
	Consider the DG and RG Markov chains.
	Recall from~\eqref{eq:PD} and~\eqref{eq:PR} that their Mtks are as follows: $\PD_{2s-1} = (P_1P_2)^{s-1}P_1$ and $\PD_{2s} = (P_1P_2)^s$ for any positive integer~$s$, while
	$\PR_t = [(1-r)P_1+rP_2]^t$ for any positive integer~$t$.
	Note that for $t \geq 2$, $\PD_t \PD_t^* = (P_1P_2P_1)^{t-1}$.
	In light of Proposition~\ref{pro:rate} and Lemma~\ref{lem:At},
	\[
	\rhoD = \|P_1P_2P_1\|^{1/2}, \quad \rhoR(r) = \|(1-r)P_1 + rP_2\|.
	\]
	Let $f \in L_0^2(\pi)$ be decomposed as in~\eqref{eq:fdecomp}.
	That is, $f = \sum_{i=0}^1 \sum_{j=0}^1 f_{ij} + f_{0} + f_{1}$.
	Assume that $M_{00} = \{0\}$ and $\MRspace \neq \{0\}$.
	One can use Lemma~\ref{lem:halmos} to obtain
	$P_1P_2P_1 f = C^2 f_{0}$.
	Then $\|P_1P_2P_1\| = \|C\|^2$.
	On the other hand, 
	\[
	[(1-r) P_1 + r P_2] f = (1-r) f_{01} + r f_{10} + [(1-r)P_1|_{\MRspace} + rP_2|_{\MRspace}] (f_{0} + f_{1}).
	\]
	It then follows from Lemma~\ref{lem:structure} that
	\[
	\begin{aligned}
		\|(1-r) P_1 + r P_2\| &= \max \left\{ 1-r, \; r, \; \frac{\|I_0 + \sqrt{(1-2r)^2 + 4r(1-r) C^2}\|}{2}  \right\} \\
		&= \max \left\{ 1-r, \; r, \; \frac{1 + \sqrt{(1-2r)^2 + 4r(1-r) \|C\|^2}}{2}  \right\} \\
		&= \frac{1 + \sqrt{(1-2r)^2 + 4r(1-r) \|C\|^2}}{2} .
	\end{aligned}
	\]
	Note that the second equality holds since $0 \leq \|C\| \leq I_0$, and 
	\[
	x \mapsto \frac{1+ \sqrt{(1-2r)^2 + 4r(1-r)x^2}}{2}
	\]
	is an increasing function on $[0,1]$.
	In summary, we have obtained the following result, which is \pcite{qin2021convergence} Theorem 4.1.
	\begin{corollary} \citep{qin2021convergence} \label{cor:rates}
		Assume that $M_{00} = \{0\}$ and $\MRspace \neq \{0\}$.
		Then
		\[
		\rhoD = \|C\|, \quad \rhoR(r) = \frac{1 + \sqrt{(1-2r)^2 + 4r(1-r) \|C\|^2}}{2} .
		\]
	\end{corollary}
	
	\begin{remark}
		The formula $\rhoD= \|C\|$ can be viewed as a special case of well-known results concerning the convergence rates of alternating projections \citep{kayalar1988error,badea2012rate}.
	\end{remark}
	
	%	As seen in the Supplement, the method herein can also be applied to other variants of two-component Gibbs.
	Corollary~\ref{cor:rates} allows for a comparison between $\rhoR(r)$ and $\rhoD$.
	To adjust for computation time, one should raise a convergence rate to the power of the number of iterations that the associated algorithm can simulate in unit time.
	To be precise, assume as before that one iteration of the DG and RG algorithms take $(1+\tau)/2$ and $r\tau + 1-r$ units of time respectively, and let
	\[
	\rhoD^{\dagger} = \rhoD^{2/(1+\tau)}, \quad \rhoR(r)^{\dagger} = \rhoR(r)^{1/(r\tau+1-r)}.
	\] 
	By Corollary~\ref{cor:rates} and Young's inequality, for $r \in (0,1)$,
	\[
	\rhoR(r)^{\dagger} \geq \|C\|^{2r(1-r)/(r\tau+1-r)} \geq \|C\|^{2/(1+\tau)} = \rhoD^{\dagger}. 
	\]
	Moreover, equality holds only if $\|C\|=1$.
	It is in this sense that the DG chain has a faster rate of convergence than the RG chain, after adjusting for computational cost.
	It is worth mentioning that this does not imply that 
	\[
	\|\mu \PD_{\lfloor 2t/(1+\tau) \rfloor} - \pi \|_* \leq \|\mu \PR_{\lfloor t/(r\tau +1 -r) \rfloor} - \pi\|_*
	\]
	for every $\mu \in L_*^2(\pi)$ and~$t$ large enough.
	That is, for certain starting distributions~$\mu$, it is still possible for the RG chain to approach~$\pi$ faster than the DG chain.
	
%	We can compare $\rhoD^{\dagger}$ and $\rhoR(r)^{\dagger}$ when~$r$ is given by~\eqref{eq:r}.
%	Let us consider only the cases $\tau \gg 1$ and $\tau \ll 1$, as these are the situations where the RG sampler can significantly outperform DG in terms of adjusted asymptotic variance.
%	Suppose that $\tau \gg 1$.
%	Then $r = \sqrt{2/\tau} + o(\tau^{-1/2})$.
%	It's straightforward to verify that given $\|C\|$,
%	\[
%	\rhoD^{\dagger} = 1 + \frac{2 \log \|C\|}{\tau}  + o(\tau^{-1}), \quad \rhoR(r)^{\dagger} = 1 - \frac{1-\|C\|^2}{\tau} + o(\tau^{-1}).
%	\]
%	If $\tau \ll 1$, $r = 1 - \sqrt{2\tau} + o(\tau^{1/2})$, and
%	\[
%	\rhoD^{\dagger} = \|C\|^2 + o(1), \quad \rhoR(r)^{\dagger} = \exp(\|C\|^2-1) + o(1).
%	\]
%	In both cases, $\rhoD^{\dagger}$ and $\rhoR(r)^{\dagger}$ are comparable when $\|C\|$, the maximal correlation between $X_1$ and $X_2$, is not close to~0.

	Finally, consider the modified DG algorithm with~$\ell$ consecutive draws from~$\pi_1$, as described in Section~\ref{ssec:modified}.
	Denote its convergence rate by $\rhoM(\ell)$.
	For $t = (\ell+1)s+q$, where~$s$ is a non-negative integer, and $q \in \{0,\dots,\ell\}$, the $t$-step Mtk of the modified DG chain is $\PM_{(\ell+1)s+q} = (P_1P_2)^s$ if $q=0$, and $\PM_{(\ell+1)s+q} = (P_1P_2)^s P_1$ if $q \in \{1,\dots,\ell\}$.
	Using Lemma~\ref{lem:halmos}, Proposition~\ref{pro:rate}, and Lemma~\ref{lem:At}, one can then establish the following.
	\begin{proposition}
		Assume that $M_{00} = \{0\}$ and $\MRspace \neq \{0\}$.
		Then
		\[
		\rhoM(\ell) = \|C\|^{2/(\ell+1)}.
		\]
	\end{proposition}
	When there is no post-processing and parallelization, it takes $(\ell + \tau)/(\ell+1)$ units of time to simulate one step of the chain.
	To adjust for computation cost, let $\rhoM^{\dagger}(\ell) = \rhoM(\ell)^{(\ell+1)/(\ell+\tau)}$.
	Then we have the following comparison result:
	\[
	\rhoR^{\dagger}\left( \frac{1}{\ell+1} \right) \geq \|C\|^{2\ell/[(\ell+1)(\ell+\tau)]} \geq \rhoM^{\dagger}(\ell) =  \|C\|^{2/(\ell+\tau)} \geq \|C\|^{2/(1+\tau)} = \rhoD^{\dagger}. 
	\]
	This, along with Corollary~\ref{cor:VMVR-1}, shows that if $r = 1/(\ell+1)$, the modified DG algorithm is better than, but comparable to the RG algorithm in terms of adjusted asymptotic variance and convergence rate.
%	(An RG sampler with a different selection probability could still perform better.)
	The standard DG algorithm is the best out of the three algorithms in terms of convergence rate, but least robust in terms of asymptotic variance.
	
	If post-processing and parallelization is employed, then, in the best case scenario, the modified DG algorithm would have the same convergence rate as the standard DG algorithm.
	Indeed, with full parallelization, the per-iteration cost of the modified DG algorithm is $(\tau + 1)/(\ell+1)$.
	In this case, $\rhoM^{\dagger}(\ell) = \rhoM(\ell)^{(\ell+1)/(\tau+1)} = \|C\|^{2/(\tau+1)}$, so $\rhoM^{\dagger}(\ell) = \rhoD^{\dagger}$.

	\section{Discussion} \label{sec:discuss}

	The methods herein can be used to analyze other variants of two-component Gibbs samplers.
	Indeed, let $(\tilde{X}_t)$ be a Markov chain with $(t',t'+t)$-Mtk $K_{t',t'+t}$.
	Suppose that for non-negative integers~$t$ and~$t'$, the operator $K_{t',t'+t}$ is in an algebra of finite type generated by $P_1$, $P_2$, and the identity operator~$I$.
	%	Suppose that for non-negative integers~$t$ and~$t'$, $K_{t',t'+t}$ is a polynomial of $P_1$ and $P_2$, that is, 
	%	\[
	%	K_{t',t'+t} = \sum_{s=0}^{n(t,t')} \left[ a(s,t,t') (P_1P_2)^s + b(s,t,t') (P_2P_1)^s + c(s,t,t') (P_1P_2)^sP_1 + d(s,t,t') P_2(P_1P_2)^s \right], 
	%	\]
	%	where $n(t,t')$ is a non-negative integer, and $a(s,t,t')$, $b(s,t,t')$, $c(s,t,t')$, $d(s,t,t')$ are non-negative numbers.
	Then $K_{t',t'+t}$ leaves $M_{00}$, $M_{01}$, $M_{10}$, $M_{11}$ and $\MRspace$ invariant.
	Moreover, $K_{t',t'+t}|_{\MRspace}$ has a matrix representation in accordance with Lemma~\ref{lem:halmos}.
	In principle, one can analyze the convergence rate and asymptotic variance of the associated MCMC algorithm through elementary matrix calculations.
	In the Supplement, this idea is applied to a random sequence scan Gibbs sampler that has Mtk $(P_1P_2 + P_2P_1)/2$.
	There are many other interesting variants of two-component Gibbs samplers that can be studied in future works.
	It'd be interesting to know which variant is optimal in terms of asymptotic variance and/or convergence rate.
	
	In applications of MCMC, one often needs to evaluate $\pi f = \int_{\X} f(x) \pi(\df x)$ for some vector-valued $f: \X \to \mathbb{R}^p$, where~$p$ is some positive integer \citep{vats2019multivariate}. 
	Based on a Markov chain $(\tilde{X}_t)$, the sample mean $S_T(f) = T^{-1} \sum_{t=0}^{T-1} f(\tilde{X}_t)$ is then a random vector.
	Under regularity conditions, the multivariate central limit theorem holds:
	\[
	\sqrt{T} [S_T(f) - \pi f] \xrightarrow{d} \mbox{N}_p(0,\tilde{V}(f)) \quad \text{as } T \to \infty,
	\]
	where $\tilde{V}(f)$ is an asymptotic covariance matrix.
	%	Let $a \in \mathbb{R}^p$.
	%	Then the display above implies that
	%	\[
	%	\sqrt{T} [a^{\top} S_T(f) - a^{\top} \pi f] \xrightarrow{d} \mbox{N}_p(0, a^{\top}V(f))a \quad \text{as } T \to \infty,
	%	\]
	If we use $\tilde{V}_{\scriptsize\mbox{D}}(f)$ and $\tilde{V}_{\scriptsize\mbox{R}}(f,r)$ to denote the asymptotic covariance matrices associated with the DG and RG samplers, respectively.
	Then for $a \in \mathbb{R}^p$,
	\[
	\VD(a^{\top}f) = a^{\top}\tilde{V}_{\scriptsize\mbox{D}}(f) a, \quad \VR(a^{\top} f, r) = a^{\top} \tilde{V}_{\scriptsize\mbox{R}}(f,r) a.
	\]
	Thus, Proposition~\ref{pro:VDVR} provides a Loewner ordering for $\VD(f)$ and $\VR(f)$.
	That is,
	\[
	\tilde{V}_{\scriptsize\mbox{D}}(f) \leq k_1(r) \tilde{V}_{\scriptsize\mbox{R}}(f,r), \quad \tilde{V}_{\scriptsize\mbox{R}}(f,r) \leq k_2(r) \tilde{V}_{\scriptsize\mbox{D}}(f),
	\]
	where $k_1$ and $k_2$ are defined in the said proposition.

	One future research avenue is to study the efficiency of the RG sampler with $r \neq 1/2$ and the modified DG sampler with $\ell > 1$ when there is a large discrepancy in the variability of $X_1$ and $X_2$.
	For instance, assume that we wish to estimate $\pi f$ and $\pi g$ where $f(x_1, x_2)$ depends only on $x_1$ and $g(x_1, x_2)$ depends only on $x_2$.
	If the variance of $f(X_1, X_2)$ is much larger than that of $g(X_1, X_2)$, how much can be gained from more frequent updates of $X_1$?

	A natural question is whether the techniques herein can be used to study Gibbs sampler with more than two components.
	Gibbs samplers with~$n$ components are associated with Markov operators generated by~$n$ orthogonal projections.
	Unfortunately, results like Lemma~\ref{lem:halmos} do not extend to the $n \geq 3$ case in general.
	See Section~11 of \cite{bottcher2010gentle} for a review.
	Whether it is possible to extend results in this work to the $n \geq 3$ case is a subject for future work.
	See \cite{greenwood1998information}, \cite{andrieu2016random}, \cite{roberts2016suprising}, and the recent work \cite{chlebicka2023solidarity} for studies on this topic based on other tools.

	\vspace{2cm}
	\appendix

	{\noindent \LARGE \bf Appendix}

	\section{Technical Results}
	
	\begin{lemma} \label{lem:k1k2}
		Let $k_1(r)$ and $k_2(r)$ be defined as in Proposition~\ref{pro:VDVR}, i.e.,
		\[
		\begin{aligned}
			k_1(r) &= 1-r+r^2 + \sqrt{r^2(1-r)^2 + (1-2r)^2}, \\
			k_2(r) &= \frac{1-r+r^2}{2r(1-r)} + \frac{\sqrt{r^2(1-r)^2 + (1-2r)^2}}{2r(1-r)}.
		\end{aligned}
		\]
		Then, for $r \in (0,1)$, each of the following holds.
		\begin{center}
			\begin{tabular}{ccc}
				(i) $k_1(r) \geq 1$; & (ii) $k_1(r) > 2r/(2-r)$; & (iii) $k_1(r) > 2(1-r)/(1+r)$; \\
				(iv) $k_2(r) \geq 2$; & (v) $k_2(r) > (2-r)/(2r)$; & (vi) $k_2(r) > (1+r)/[2(1-r)]$.
			\end{tabular}
		\end{center}
		%		\begin{enumerate}
			%			\item [(i)] $k_1(r) \geq 1$;
			%			\item [(ii)] $k_1(r) > 2r/(2-r)$;
			%			\item [(iii)] $k_1(r) > 2(1-r)/(1+r)$;
			%			\item [(iv)] $k_2(r) \geq 2$;
			%			\item [(v)] $k_2(r) > (2-r)/(2r)$;
			%			\item [(vi)] $k_2(r) > (1+r)/[2(1-r)]$.
			%		\end{enumerate}
	\end{lemma}
	\begin{proof}
		It's obvious that
		\[
		k_1(r) \geq 1-r+r^2 + |r-r^2|,
		\]
		so (i) holds.
		Moreover,
		\[
		k_1(r) \geq 1-r+r^2 + |1-2r|.
		\]
		Thus,
		\[
		k_1(r) \geq r^2 + r = r(1+r) > \frac{2r}{2-r},
		\]
		and
		\[
		k_1(r) \geq 2 - 3r + r^2 = (1-r)(2-r) > \frac{2(1-r)}{1+r}.
		\]
		This proves (i)-(iii).
		
		The proofs for (iv)-(vi) are similar.
	\end{proof}

	\begin{lemma} \label{lem:positive}
		Let $A$ be a linear operator on $\MRspace$ with matrix representation
		\[
		A = \Gamma^* \left( \begin{array}{cc}
			d I_0 + a C^2S^{-2} & b CS^{-1} \\
			b CS^{-1} & c I_0
		\end{array} \right)  \Gamma, \quad a,b,c,d \in \mathbb{R}.
		\]
		Then~$A$ is positive semi-definite if each of the following conditions holds:
		\begin{enumerate}
			\item[(i)] $a, c, d \geq 0$;
			\item[(ii)] $ac - b^2 \geq 0$.
		\end{enumerate}
	\end{lemma}
	\begin{proof}
		If $c = 0$, then $(ii)$ implies that $b = 0$, and the proof is trivial.
		Assume that $c > 0$.
		Then $A \geq B$, where $\geq$ denotes Loewner ordering, and
		\[
		B = \Gamma^* \left( \begin{array}{cc}
			\frac{b^2}{c} C^2S^{-2} & b CS^{-1} \\
			b CS^{-1} & c I_0
		\end{array} \right) \Gamma.
		\]
		Recall that $\Gamma = I_0 \oplus W$ where $W: (I-P_1)\MRspace \to \MRspace$ is unitary.
		For $f_{0} \in P_1 \MRspace$ and $f_{1} \in (I - P_1)\MRspace$,
		\[
		\begin{aligned}
			\langle f_{0} + f_{1}, B (f_{0} + f_{1}) \rangle 
			=&  \frac{b^2}{c} \langle f_{0}, C^2S^{-2} f_{0} \rangle + b \langle f_{0}, CS^{-1} W f_{1} \rangle + b \langle f_{1}, W^* CS^{-1} f_{0} \rangle + c\| f_{1}\|^2 \\
			=& \left\| \frac{b}{\sqrt{c}} CS^{-1} f_{0} + \sqrt{c} W f_{1} \right\|^2 \\
			\geq & 0.
		\end{aligned}
		\]
		Thus, $B$ is positive semi-definite, and so is~$A$.
	\end{proof}

	\begin{lemma} \label{lem:SDSR}
		For $r \in (0,1)$, let $\SD$ and $\SR(r)$ be defined as in~\eqref{eq:SD} and~\eqref{eq:SR}, and let $k_1(r)$ and $k_2(r)$ be defined as in Lemma~\ref{lem:k1k2}.
		Then $\SD  \leq k_1(r) \SR(r)$, and $\SR(r) \leq k_2(r) \SD$, where $\leq$ denotes Loewner ordering.
	\end{lemma}
	
	\begin{proof}
		I will establish $\SD \leq k_1(r) \SR(r)$.
		The other inequality can be proved in a similar way.
		\[
		k_1(r)\SR(r) - \SD = \Gamma^* \left( \begin{array}{cc}
			d(r) I_0 + a(r) C^2S^{-2} & b(r) CS^{-1} \\
			b(r) CS^{-1} & c(r) I_0
		\end{array} \right) \Gamma,
		\]
		where
		\[
		\begin{aligned}
			a(r) &= \frac{2}{r(1-r)} k_1(r) - 4, \\
			b(r) &= \frac{2}{1-r} k_1(r) - 2, \\
			c(r) &= \frac{1+r}{1-r} k_1(r) - 2, \\
			d(r) &= \frac{2-r}{r} k_1(r) - 2.
		\end{aligned}
		\]
		By Lemma~\ref{lem:k1k2}, $a(r), c(r), d(r) \geq 0$.
		Straightforward calculations reveal that
		\begin{equation} \nonumber
			a(r)c(r) - b(r)^2 = \frac{2}{r(1-r)}k_1(r)^2 - \frac{4(r^2-r+1)}{r(1-r)} k_1(r) + 4 = 0.
		\end{equation}
		It follows from Lemma~\ref{lem:positive} that $\SD \leq k_1(r) \SR(r)$.
	\end{proof}
	
	The following result can be proved in a similar fashion.
	\begin{lemma} \label{lem:SM}
	Let $\SD$, $\SR(r)$, $\SM(\ell)$ be defined as in~\eqref{eq:SD},~\eqref{eq:SR}, and~\eqref{eq:SM}.
	Then, for $\ell \geq 1$, 
	\[
	\SM(\ell) \leq \SR \left( \frac{1}{\ell+1} \right) \leq \left\{ \frac{2\ell+1}{\ell+1} ,\frac{\ell+1}{\ell} \right\} \SM(\ell),
	\]
	and 
	\[
	\begin{aligned}
		\SM(\ell) &\leq \frac{\ell^2+\ell+2+(\ell-1) \sqrt{(\ell+1)^2+1}}{2(\ell+1)} \SD, \\
		\SD &\leq \frac{\ell^2+\ell+2+(\ell-1)\sqrt{(\ell+1)^2+1}}{(\ell+1)^2} \SM(\ell).
	\end{aligned}
	\]
	\end{lemma}
	
	In Section~\ref{sec:main} it was claimed that the bounds in Proposition~\ref{pro:VDVR} are sharp.
	The following proposition makes this precise.
	\begin{proposition} \label{pro:bound-sharp}
		Suppose that $M_{00} = \{0\}$ and $\MRspace \neq \{0\}$.
		Let $r \in (0,1)$ be given, and let $k_1(r)$, $k_2(r)$ be defined as in Lemma~\ref{lem:k1k2}.
		Then, for any $\eta < k_1(r)$, if $\|C\|$ is smaller than but sufficiently close to~1, one can find a non-zero $f \in L_0^2(\pi)$ such that
		\[
		\VD(f) > \eta \VR(f,r).
		\]
		Moreover, for any $\eta < k_2(r)$, if $\|C\|$ is smaller than but sufficiently close to~1, one can find a non-zero $f \in L_0^2(\pi)$ such that
		\[
		\VR(f,r) > \eta \VD(f).
		\]
	\end{proposition}
	\begin{proof}
		I'll prove the first assertion.
		The proof of the second is similar.
		
		Let $\eta < k_1(r)$ be arbitrary.
		For $k > 0$, define
		\[
		\begin{aligned}
			\alpha(k) &= \frac{2}{r(1-r)} k - 4, \\
			\beta(k) &= \frac{2}{1-r} k - 2, \\
			\gamma(k) &= \frac{1+r}{1-r} k - 2.
		\end{aligned}
		\] 
		Then
		\[
		\alpha(k)\gamma(k) - \beta(k)^2 = \frac{2}{r(1-r)} k^2 - \frac{4(r^2-r+1)}{r(1-r)} k + 4.
		\]
		$k=k_1(r)$ is the larger root of the equation $\alpha(k)\gamma(k) - \beta(k)^2 = 0$, which always has two roots.
		Moreover, by Lemma~\ref{lem:k1k2}, $\gamma(k_1(r)) > 0$.
		Without loss of generality, assume that~$\eta$ is sufficiently close to $k_1(r)$ so that
		\[
		\alpha(\eta)\gamma(\eta) - \beta(\eta)^2 < 0, \quad \gamma(\eta) > 0.
		\]
		One can show that
		\begin{equation} \label{eq:ellS}
			\begin{aligned}
				&\eta \SR(r) - \SD \\&= \Gamma^* \left( \begin{array}{cc}
					\frac{\beta(\eta)^2}{\gamma(\eta)} C^2S^{-2} & \beta(\eta) CS^{-1} \\
					\beta(\eta) CS^{-1} & \gamma(\eta) I_0
				\end{array} \right) \Gamma + \Gamma^* \left( \begin{array}{cc}
					\left(\frac{2-r}{r} \eta - 2 \right) I_0 +  \left[ \alpha(\eta) - \frac{\beta(\eta)^2}{\gamma(\eta)} \right] C^2S^{-2} & 0 \\
					0 & 0
				\end{array} \right) \Gamma.
			\end{aligned}
		\end{equation}
		Basic properties regarding continuous functions of self-adjoint operators imply that
		\[
		\|C^2S^{-2}\| = \|C^2(I_0-C^2)^{-1}\| = \|C\|^2/(1-\|C\|^2).
		\]
		Thus as $\|C\| \to 1$ from below, $\|C^2S^{-2}\| \to \infty$.
		Moreover, $C^2S^{-2}$ is positive semi-definite.
		When $\|C\|$ is sufficiently close to~1, one can find a non-constant function $f_{0} \in P_1\MRspace$ such that
		\[
		\left(\frac{2-r}{r} \eta - 2 \right) \|f_{0}\|^2 +  \left[ \alpha(\eta) - \frac{\beta(\eta)^2}{\gamma(\eta)} \right] \langle f_{0}, C^2S^{-2} f_{0} \rangle < 0.
		\]
		Let $f = f_{0} - \beta(\eta)\gamma(\eta)^{-1} W^* CS^{-1} f_{0}$, where $W^*$ is the lower right block of $\Gamma^*$.
		Then by Propositions~\ref{pro:VD} and~\ref{pro:VR} along with~\eqref{eq:ellS},
		\[
		\begin{aligned}
			\eta \VR(f,r) - \VD(f) &= \langle f, [\eta \SR(r) - \SD ] f \rangle \\
			&= \left\langle f_{0}, \left[\frac{2-r}{r} \eta - 2 \right] f_{0} +  \left[ \alpha(\eta) - \frac{\beta(\eta)^2}{\gamma(\eta)} \right] C^2S^{-2} f_{0} \right\rangle \\
			& < 0.
		\end{aligned}
		\]
	\end{proof}
	
	\section{Proof of Proposition~\ref{pro:rate}} \label{app:rate}

	Let $\mu \in L_*^2(\pi)$ be such that $\mu \neq \pi$, and set $g = \df \mu/\df \pi - 1 \in L_0^2(\pi)$.
	Then, by~\eqref{eq:l2dist},
	\[
	\begin{aligned}
		\limsup_{t \to \infty} t^{-1} \log \|\mu K_t - \pi\|_* &= \log \left( \limsup_{t \to \infty} \|g\|^{1/t} \langle g/\|g\|, K_t K_t^* g/\|g\| \rangle^{1/(2t)} \right) \\
		&= \log \left( \limsup_{t \to \infty} \langle g/\|g\|, K_t K_t^* g/\|g\| \rangle^{1/(2t)} \right).
	\end{aligned}
	\]
	It follows that
	\[
	\sup_{\mu \in L_*^2(\pi),\; \mu \neq \pi} \limsup_{t \to \infty} t^{-1} \log \|\mu K_t - \pi\|_* \leq  \log \left( \sup_{f \in L_0^2(\pi), \; \|f\|=1} \limsup_{t \to \infty} \langle f, K_t K_t^* f \rangle^{1/(2t)} \right).
	\]
	
	To prove the reverse inequality, let $f \in L_0^2(\pi)$ be such that $\|f\|=1$, and decompose it into positive and negative parts: $f = f_+ - f_-$.
	Then
	\[
	c := \pi f_+ = \pi f_- > 0.
	\]
	Let $\mu_+$ and $\mu_-$ be probability measures such that $\df \mu_+/\df \pi = c^{-1} f_+$ and $\df \mu_-/\df \pi = c^{-1} f_-$.
	Then $\mu_+, \mu_- \in L_*^2(\pi)$, and neither of the two is equal to~$\pi$.
	By~\eqref{eq:l2dist},
	\[
	\begin{aligned}
		\langle f, K_tK_t^* f \rangle^{1/2} &= c\|K_t^* [(c^{-1}f_+-1) - (c^{-1}f_--1)] \| \\
		&\leq c\| K_t^* (c^{-1}f_+ - 1)\| + c\|K_t^*(c^{-1}f_- - 1)\| \\
		&= c\|\mu_+ K_t - \pi\|_* + c\|\mu_- K_t - \pi\|_*.
	\end{aligned}
	\]
	For $a,b > 0$, $\log (a+b) \leq \log 2 + \max\{\log a, \, \log b\}$.
	It follows that 
	\[
	\log \left( \langle f, K_tK_t^* f \rangle^{1/(2t)} \right) \leq t^{-1} \log (2c) +  \max \left\{ t^{-1} \log \|\mu_+ K_t - \pi\|_*,\; t^{-1} \log \|\mu_- K_t - \pi\|_*  \right\}
	\]
	For two sequences of real numbers $(a_t)$ and $(b_t)$, 
	\[
	\limsup_{t \to \infty} \max\{a_t,\, b_t\} = \max\left\{ \limsup_{t \to \infty} a_t,\, \limsup_{t \to \infty} b_t \right\}.
	\]
	Thus,
	\[
	\begin{aligned}
		\log \left( \limsup_{t \to \infty} \langle f, K_t K_t^* f \rangle^{1/(2t)} \right) &= \limsup_{t \to \infty} \log \left( \langle f, K_t K_t^* f \rangle^{1/(2t)} \right) \\
		&\leq \sup_{\mu \in L_*^2(\pi),\; \mu \neq \pi} \limsup_{t \to \infty} t^{-1} \log \|\mu K_t - \pi\|_*.
	\end{aligned}
	\]
	Taking supremum with respect to $f$ shows that
	\[
	\log \left( \sup_{f \in L_0^2(\pi), \; \|f\|=1} \limsup_{t \to \infty} \langle f, K_t K_t^* f \rangle^{1/(2t)} \right) \leq \sup_{\mu \in L_*^2(\pi),\; \mu \neq \pi} \limsup_{t \to \infty} t^{-1} \log \|\mu K_t - \pi\|_*.
	\]

	\bigskip
	
	%%%%%%%%%%%%%%%%%%%%%%%%%%%%%%%%%%%%%%%%%%%%%%
	%% Support information, if any,             %%
	%% should be provided in the                %%
	%% Acknowledgements section.                %%
	%%%%%%%%%%%%%%%%%%%%%%%%%%%%%%%%%%%%%%%%%%%%%%
	{\bf Acknowledgments.}
		The author was partially supported by NSF Grant DMS-2112887.
		The author thanks an Editor, an anonymous Associate Editor, and two anonymous referees for their valuable feedback.
		In particular, the referees suggested studying the modified DG sampler and its parallelization.
		The author would like to thank Riddhiman Bhattacharya, Austin Brown, James P. Hobert, Galin L. Jones, and Haoxiang Li for their helpful comments.

	%%%%%%%%%%%%%%%%%%%%%%%%%%%%%%%%%%%%%%%%%%%%%%
	%% Funding information, if any,             %%
	%% should be provided in the                %%
	%% funding section.                         %%
	%%%%%%%%%%%%%%%%%%%%%%%%%%%%%%%%%%%%%%%%%%%%%%

	\bibliographystyle{ims}
	\bibliography{qinbib}

	\newpage
	
	% {\noindent\bf\LARGE Supplement}
	
	\setcounter{section}{18}
	
	\section{Supplement}
	
	\subsection{Proof of Lemma~\ref{lem:structure}}

	Recall that the lemma states that for $r \in (0,1)$,
	\[
	(1-r) P_1|_{\MRspace} + rP_2|_{\MRspace} = \Gamma^* U \left( \begin{array}{cc}
		[I_0+\sqrt{\Delta(C)}]/2 & 0 \\
		0 & [I_0-\sqrt{\Delta(C)}]/2
	\end{array} \right) U^* \Gamma,
	\]
	where $U: \MRspace \to \MRspace$ is unitary, and
	\[
	\Delta(C) = (1-2r)^2 I_0 + 4r(1-r) C^2.
	\]
	
	Consider first the case $r=1/2$.
	When this is the case, 
	\[
	(1-r) P_1|_{\MRspace} + rP_2|_{\MRspace} = \Gamma^* \left( \begin{array}{cc}
		I_0/2+C^2/2 & CS/2 \\
		CS/2 & S^2/2
	\end{array} \right) \Gamma,
	\]
	$\Delta(C) = C^2$, and the result holds with
	\[
	U =  \left( \begin{array}{cc}
		\sqrt{(I_0+C)/2} & \sqrt{(I_0-C)/2} \\
		\sqrt{(I_0-C)/2} & - \sqrt{(I_0+C)/2}
	\end{array} \right).
	\]
	
	Assume that $r \neq 1/2$.
	For $x \in (0,1)$, define
	\[
	\Delta(x) = (1-2r)^2 + 4r(1-r)x^2.
	\]
	Also, for $x \in (0,1)$, let
	\[
	\begin{aligned}
		h_{00}(x) &= \frac{[1+\sqrt{\Delta(x)}]x}{\sqrt{ 4(1-r)[r+\sqrt{\Delta(x)}]x^2 + [2r-1+\sqrt{\Delta(x)}]^2} }, \\
		h_{10}(x) &= \frac{[2r-1+\sqrt{\Delta(x)}] \sqrt{1-x^2}}{\sqrt{ 4(1-r)[r+\sqrt{\Delta(x)}]x^2 + [2r-1+\sqrt{\Delta(x)}]^2} }, \\
		h_{01}(x) &= \frac{[1-\sqrt{\Delta(x)}]x}{\sqrt{ 4(1-r) [r-\sqrt{\Delta(x)}] x^2 + [2r-1-\sqrt{\Delta(x)}]^2 }}, \\
		h_{11}(x) &= \frac{ [2r-1-\sqrt{\Delta(x)}] \sqrt{1-x^2}} {\sqrt{ 4(1-r) [r-\sqrt{\Delta(x)}] x^2 + [2r-1-\sqrt{\Delta(x)}]^2 }}.
	\end{aligned}
	\]
	When $r<1/2$, let $h_{00}(0) = 1$, $h_{10}(0) = 0$, $h_{01}(0)=0$, $h_{11}(0)=-1$, and let $h_{00}(1) = 1$, $h_{10}(1) = 0$, $h_{01}(1)=0$, $h_{11}(1) = -1$.
	When $r>1/2$, let $h_{00}(0) = 0$, $h_{10}(0) = 1$, $h_{01}(0)=1$, $h_{11}(0)=0$, and let $h_{00}(1) = 1$, $h_{10}(1) = 0$, $h_{01}(1) = 0$, $h_{11}(1) = -1$.
	Then, for $i \in \{0,1\}$, $h_{ij}$ is continuous on $[0,1]$.
	For $r \in (0,1)$ and $x \in [0,1]$, the matrix
	\[
	u(x) = \left( \begin{array}{cc}
		h_{00}(x) & h_{01}(x) \\
		h_{10}(x) & h_{11}(x)
	\end{array} \right)
	\]
	is unitary since $u(x)^{\top} u(x) = I$.
	Moreover,
	\[
	\left( \begin{array}{cc}
		1-r + rx^2 & rx\sqrt{1-x^2} \\
		rx\sqrt{1-x^2} & r(1-x^2)
	\end{array} \right) u(x) = u(x) \left(\begin{array}{cc}
		[1+\sqrt{\Delta(x)}]/2 & 0 \\
		0 & [1-\sqrt{\Delta(x)}]/2
	\end{array} \right),
	\]
	so
	\[
	\left( \begin{array}{cc}
		1-r + rx^2 & rx\sqrt{1-x^2} \\
		rx\sqrt{1-x^2} & r(1-x^2)
	\end{array} \right) = u(x) \left(\begin{array}{cc}
		[1+\sqrt{\Delta(x)}]/2 & 0 \\
		0 & [1-\sqrt{\Delta(x)}]/2
	\end{array} \right) u(x)^{\top}.
	\]
	
	The spectrum of~$C$ is a subset of $[0,1]$.
	Using polynomial approximation \citep[see, e.g.,][\S 32]{helmberg2014introduction}, one can define $h_{ij}(C)$ for $i,j \in \{1,2\}$.
	Let
	\[
	U = \left( \begin{array}{cc}
		h_{00}(C) & h_{01}(C) \\
		h_{10}(C) & h_{11}(C)
	\end{array} \right).
	\]
	For any two real continuous functions $h_1$ and $h_2$ on $[0,1]$, it holds that $h_1(C) + h_2(C) = (h_1+h_2)(C)$, and $h_1(C) h_2(C) = (h_1h_2)(C)$.
	By previous calculations, one can then conclude that~$U$ is unitary, and that Lemma~\ref{lem:structure} holds.

	\subsection{Another Variant of Two-component Gibbs}
	
	One can use the theory of two projections to analyze other variants of two-component Gibbs samplers.
	As a demonstration, consider the following random sequence scan Gibbs sampler.
	
	\bigskip
	
	\begin{algorithm}[H]
		\caption{Random sequence scan Gibbs sampler}
		Draw $\tilde{X}_0 = (X_{1,0},X_{2,0})$ from some initial distribution on $(\X, \B)$, and
		set $t=0$ \;
		
		\While{$t < T$}{
			draw $W$ from a $\mbox{Bernoulli}(1/2)$ distribution\;
			\If{$W=0$}{
				draw $X_{2,t+1}$ from $\pi_1(\cdot\mid X_{1,t})$, then draw $X_{1,t+1}$ from $\pi_2(\cdot\mid X_{2,t+1})$;
			}
			\If{$W=1$}{
				Draw $X_{1,t+1}$ from $\pi_2(\cdot\mid X_{2,t})$, then draw $X_{2,t+1}$ from $\pi_1(\cdot\mid X_{1,t+1})$;
			}
			set $\tilde{X}_{t+1} = (X_{1,t+1},X_{2,t+1})$\;
			set $t=t+1$\;
		}
	\end{algorithm}
	
	\bigskip
	
	The underlying Markov chain $\tilde{X}_t$ is time-homogeneous, and the associated $t$-step Mtk is
	\[
	\PS_t = (P_1P_2/2 + P_2 P_1/2)^t.
	\]
	It is clear that $\pi \PS_t = \pi$, so the chain has~$\pi$ as a stationary distribution.
	Assume that $M_{00} = \{0\}$ and that $\MRspace \neq \{0\}$.
	Then, for $f \in L_0^2(\pi)$ with orthogonal decomposition $f=\sum_{i=0}^1\sum_{j=0}^1 f_{ij} + f_{0} + f_{1}$ as in~\eqref{eq:fdecomp},
	\begin{equation} \label{eq:Sf}
		\PS_1 f = (P_1|_{\MRspace} P_2|_{\MRspace} / 2 + P_2|_{\MRspace} P_1|_{\MRspace} / 2) (f_{0} + f_{1}).
	\end{equation}
	By Lemma~\ref{lem:halmos}, the following matrix representation holds:
	\begin{equation} \label{eq:HMatrix}
		P_1|_{\MRspace} P_2|_{\MRspace} / 2 + P_2|_{\MRspace} P_1|_{\MRspace} / 2 = \Gamma^* \left( \begin{array}{cc}
			C^2 & CS/2 \\
			CS/2 & 0
		\end{array} \right) \Gamma.
	\end{equation}

	For $f \in L_0^2(\pi)$ defined above, the associated asymptotic variance is 
	\[
	\VS(f) = E[f(\tilde{X}_0)^2] + 2 \sum_{t=1}^{\infty} E[f(\tilde{X}_0)  f(\tilde{X}_t)],
	\]
	where $\tilde{X}_0 \sim \pi$ \citep[see, e.g.,][]{jones2004markov}.
	
	\begin{proposition} \label{pro:VS}
 		Assume that $\|C\| < 1$.
		Then
		\[
		\VS(f) = \|f_{01}\|^2 + \|f_{10}\|^2 + \|f_{11}\|^2 +  \langle f_{0} + f_{1}, \SigS (f_{0} + f_{1}) \rangle,
		\]
		where
		\[
		\SigS= \Gamma^*  A_{\scriptsize\mbox{S}} B_{\scriptsize\mbox{S}} \Gamma,
		\]
		\[
		A_{\scriptsize\mbox{S}} = \left( \begin{array}{cc}
			(I_0-C^2/4)^{-1} & 0 \\
			0 & (I_0-C^2/4)^{-1} 
		\end{array} \right), \quad B_{\scriptsize\mbox{S}} = \left( \begin{array}{cc}
			I_0 + 2C^2S^{-2} + C^2/4 &  CS^{-1} \\
			CS^{-1} & I_0 + C^2/4
		\end{array} \right).
		\]
	\end{proposition}
	\begin{proof}
	By~\eqref{eq:Sf},
	\begin{equation} \label{eq:VS}
		\begin{aligned}
			\VS(f) =& \|f\|^2 + 2 \sum_{t=1}^{\infty} \langle f, (P_1P_2/2 + P_2 P_1/2)^t f \rangle  \\
			= & \|f_{01}\|^2 + \|f_{10}\|^2 + \|f_{11}\|^2 +  \langle f_{0} + f_{1}, \Sigma (f_{0} + f_{1}) \rangle,
		\end{aligned}
	\end{equation}
	where
	\[
	\Sigma = I + 2 \sum_{t=1}^{\infty} \left( P_1|_{\MRspace}P_2|_{\MRspace}/2 + P_2|_{\MRspace}P_1|_{\MRspace}/2 \right)^t.
	\]
	By the matrix representation,
	\begin{equation} \label{eq:SigS}
		\begin{aligned}
			\Sigma &= I + 2 \Gamma^* \sum_{t=1}^{\infty} \left( \begin{array}{cc}
				C^2 & CS/2 \\
				CS/2 & 0
			\end{array} \right)^t \Gamma \\
			&= 2 \Gamma^* \left( \begin{array}{cc}
				S^2 & - CS/2 \\
				- CS/2 & I_0
			\end{array} \right)^{-1} \Gamma - I \\
			&=  \SigS.
		\end{aligned}
	\end{equation}
	The desired result then follows from~\eqref{eq:VS}.
	\end{proof}
	
	The following lemma holds for $\SigS$.
	
	\begin{lemma} \label{lem:SigS-SD}
		Let $\SD$ and $\SigS$ be as in~\eqref{eq:SD} and~\eqref{eq:SigS}, respectively.
		Then
		\[
		\SD \leq 2\SigS.
		\]
	\end{lemma}
	
	\begin{proof}
		Recall that
		\[
		\SigS= \Gamma^*  A_{\scriptsize\mbox{S}} B_{\scriptsize\mbox{S}} \Gamma,
		\]
		where $A_{\scriptsize\mbox{S}}$ and $B_{\scriptsize\mbox{S}}$ are given in Proposition~\ref{pro:VS}.
%		where
%		\[
%		A_{\scriptsize\mbox{S}} = \left( \begin{array}{cc}
%			(I_0-C^2/4)^{-1} & 0 \\
%			0 & (I_0-C^2/4)^{-1} 
%		\end{array} \right), \quad B_{\scriptsize\mbox{S}} = \left( \begin{array}{cc}
%			I_0 + 2C^2S^{-2} + C^2/4 &  CS^{-1} \\
%			CS^{-1} & I_0 + C^2/4
%		\end{array} \right).
%		\]
		Since $0 \leq \|C\| < 1$, $A_{\scriptsize\mbox{S}} \geq I$.
		Moreover, $(A_{\scriptsize\mbox{S}} - I)^{1/2}$ commutes with $B_{\scriptsize\mbox{S}}$.
		It is straightforward to verify that $B_{\scriptsize\mbox{S}}$ is positive semi-definite.
		Then
		\[
		\SigS - \Gamma^* B_{\scriptsize\mbox{S}} \Gamma = \Gamma^* (A_{\scriptsize\mbox{S}} - I) B_{\scriptsize\mbox{S}} \Gamma = \Gamma^* (A_{\scriptsize\mbox{S}} - I)^{1/2} B_{\scriptsize\mbox{S}} (A_{\scriptsize\mbox{S}} - I)^{1/2}  \Gamma.
		\]
		Hence, $\SigS \geq \Gamma^* B_{\scriptsize\mbox{S}} \Gamma$.
		On the other hand,
		\[
		\Gamma^* B_{\scriptsize\mbox{S}} \Gamma \geq \Gamma^* \left( \begin{array}{cc}
			I_0 + 2C^2S^{-2}  &  CS^{-1} \\
			CS^{-1} & I_0
		\end{array} \right) \Gamma = \frac{1}{2} \SD.
		\]
		The desired result then follows.
	\end{proof}
	
	Recall from Proposition~\ref{pro:VD} that the asymptotic variance for the DG sampler is
	\[	
	\VD(f) = 2\|f_{01}\|^2 + 2\|f_{10}\|^2 + \|f_{11}\|^2 + \langle f_{0} + f_{1}, \SD (f_{0} + f_{1}) \rangle.
	\]
	It then follows from Proposition~\ref{pro:VS} and Lemma~\ref{lem:SigS-SD} that
	\[
	\VD(f) \leq 2\VS(f).
	\]
	On average, one iteration of the random sequence scan algorithm takes twice the time one iteration of the DG algorithm takes.
	Let $\VS^{\dagger}(f)$ be the computation time adjusted asymptotic variance of the random sequence scan algorithm, i.e., $\VS^{\dagger}(f) = (\tau + 1)\VS(f)$.
	Then
	\[
	\VD^{\dagger}(f) \leq \VS^{\dagger}(f).
	\]
	Thus, the random sequence scan sampler is no more efficient than the DG sampler in terms of adjusted asymptotic variance (and thus could be much worse than the RG sampler).
	
	Consider now the convergence rate of the chain.
	Define
	\[
	\rhoS = \exp \left( \sup_{\mu \in L_*^2(\pi),\; \mu \neq \pi} \limsup_{t \to \infty} t^{-1} \log \|\mu \PS_t - \pi \|_* \right).
	\]
	By Proposition~\ref{pro:rate} and Lemma~\ref{lem:At},
	\[
	\rhoS = \|P_1P_2/2 + P_2P_1/2\|  = \left\|P_1|_{\MRspace} P_2|_{\MRspace} /2 + P_2|_{\MRspace} P_1|_{\MRspace} /2 \right\|,
	\]
	where the second equality follows from~\eqref{eq:Sf}.
	Consider the matrix representation~\eqref{eq:HMatrix} and note that
	\[
	\left( \begin{array}{cc}
		C^2 & CS/2 \\
		CS/2 & 0
	\end{array} \right) = U_1 \left( \begin{array}{cc}
		(I_0+C)C/2 & 0 \\
		0 & -C(I_0-C)/2
	\end{array} \right) U_1^*,
	\]
	where
	\[
	U_1 = \left( \begin{array}{cc}
		\sqrt{(I_0+C)/2} & \sqrt{(I_0-C)/2} \\
		\sqrt{(I_0-C)/2} & -\sqrt{(I_0+C)/2}
	\end{array} \right)
	\]
	is unitary and $U_1^* = U_1$ is its adjoint.
	Then 
	\[
	\rhoS = \left\|P_1|_{\MRspace} P_2|_{\MRspace} /2 + P_2|_{\MRspace} P_1|_{\MRspace} /2 \right\| = (1+\|C\|)\|C\|/2.
	\]
%	Following previous assumptions on the computation time of Gibbs samplers, suppose that one iteration of the random sequence scan Gibbs chain takes $1+\tau$ units of time to simulate.
	After adjusting for computation time, the convergence rate is
	\[
	\rhoS^{\dagger} := \rhoS^{1/(1+\tau)} = [(1+\|C\|)\|C\|/2]^{1/(1+\tau)}.
	\]
	It is obvious that $\rhoS^{\dagger} \geq \rhoD^{\dagger}$, where $\rhoD^{\dagger} = \|C\|^{2/(1+\tau)}$ is given in Section~\ref{sec:rate}.
	Depending on the value of $\|C\|$,~$r$, and~$\tau$, $\rhoS^{\dagger}$ can be larger or smaller than $\rhoR^{\dagger}(r)$.

\end{document}